\documentclass[a4paper,reqno]{amsart}

\usepackage[french,british]{babel}
\usepackage{amssymb,amsmath,amsthm,amscd,mathrsfs,graphicx,color,multicol,booktabs,bm,manfnt}
\usepackage[cmtip,all,matrix,arrow,tips,curve]{xy}

\usepackage{mathabx}

\usepackage{manfnt}
\usepackage{centernot}   

\newcommand{\CC}{{\mathbb C}}
\newcommand{\cA}{{\mathscr A}}

\newcommand{\cB}{{\mathscr B}}
\newcommand{\cC}{{\mathscr C}}

\newcommand{\cE}{{\mathscr E}}
\newcommand{\cF}{{\mathscr F}}
\newcommand{\cG}{{\mathscr G}}
\newcommand{\cH}{{\mathscr H}}
\newcommand{\cI}{{\mathscr I}}
\newcommand{\cJ}{{\mathscr J}}
\newcommand{\cK}{{\mathscr K}}
\newcommand{\cL}{{\mathscr L}}
\newcommand{\cM}{{\mathscr M}}
\newcommand{\cN}{{\mathscr N}}

\newcommand{\cO}{{\mathscr O}}
  
\newcommand{\cQ}{{\mathscr Q}}

\newcommand{\cS}{{\mathscr S}}

\newcommand{\cU}{{\mathscr U}}
\newcommand{\cV}{{\mathscr V}} 

\newcommand{\cX}{{\mathscr X}} 

\newcommand{\cZ}{{\mathscr Z}}

\newcommand{\dra}{\dashrightarrow}

\newcommand{\Ext}{\text{Ext}}

\newcommand{\gquot}{/\!\!/}

\newcommand{\hra}{\hookrightarrow}

\newcommand{\la}{\langle}

\newcommand{\lra}{\longrightarrow}

\newcommand{\ov}{\overline}
\newcommand{\PP}{{\mathbb P}}
\newcommand{\QQ}{{\mathbb Q}}
\newcommand{\ra}{\rangle}
\newcommand{\RR}{{\mathbb R}}

\newcommand{\wt}{\widetilde}
\newcommand{\ZZ}{{\mathbb Z}}
\newcommand{\Gr}{\mathrm{Gr}}

\theoremstyle{plain}

\newtheorem{thm}{Theorem}[section]
\newtheorem*{thm*}{Theorem}

\newtheorem{clm}[thm]{Claim}

\newtheorem{crl}[thm]{Corollary}

\newtheorem*{hyp*}{Hypothesis}
\newtheorem{lmm}[thm]{Lemma}
\newtheorem{prp}[thm]{Proposition}
\newtheorem{prp-dfn}[thm]{Proposition-Definition}

\theoremstyle{definition}

\newtheorem{dfn}[thm]{Definition}

\theoremstyle{remark}

\newtheorem{expl}[thm]{Example}

\newtheorem*{qst*}{Main Question}
\newtheorem{rmk}[thm]{Remark}

\DeclareMathOperator{\Amp}{Amp}

\DeclareMathOperator{\Aut}{Aut}

\DeclareMathOperator{\ch}{ch}

\DeclareMathOperator{\cl}{cl}

\DeclareMathOperator{\Coh}{Coh}

\DeclareMathOperator{\coker}{coker}

\DeclareMathOperator{\Def}{Def}

\DeclareMathOperator{\Diag}{Diag}
\DeclareMathOperator{\disc}{disc}

\DeclareMathOperator{\End}{End}
\DeclareMathOperator{\ev}{ev}

\DeclareMathOperator{\Hom}{Hom}
\DeclareMathOperator{\Id}{Id}
\DeclareMathOperator{\im}{Im}

\DeclareMathOperator{\Kum}{Kum}

\DeclareMathOperator{\NS}{NS}

\DeclareMathOperator{\PGL}{PGL}

\DeclareMathOperator{\Pic}{Pic}

\DeclareMathOperator{\rk}{rk}

\DeclareMathOperator{\sing}{sing}
\DeclareMathOperator{\SL}{SL}

\DeclareMathOperator{\Sym}{Sym}
\DeclareMathOperator{\td}{Td}

\def\bw#1{\textstyle{\bigwedge\hskip-0.9mm^{#1}}}

\usepackage{multirow}

 \makeindex
 
\begin{document}
 \title{Modular sheaves on hyperk\"ahler varieties}
 \author{Kieran G. O'Grady}
 \address{Dipartimento di Matematica, 
 Sapienza Universit\`a di Roma,
 P.le A.~Moro 5,
 00185 Roma - ITALIA}
 \email{ogrady@mat.uniroma1.it}
\dedicatory{Dedicato a Titti}
\date{\today}
\thanks{Partially supported by PRIN 2015}
\begin{abstract}
A torsion free sheaf on a hyperk\"ahler manifold $X$ is  modular if its discriminant satisfies a certain condition, for example this is the case if it is a multiple of  $c_2(X)$. The definition is tailor made for torsion-free sheaves  on  a polarized hyperk\"ahler variety  $(X,h)$
 which deform to all small deformations of $(X,h)$. For hyperk\"ahler varieties  of Type $K3^{[2]}$ we prove an existence and uniqueness result for  slope-stable vector bundles with certain ranks, $c_1$ and $c_2$. As a consequence we get uniqueness up to isomorphism of the tautological quotient rank $4$ vector bundle on the variety of lines on a generic cubic $4$-dimensional hypersurface, and on the Debarre-Voisin variety associated to a generic  element of $\bigwedge^3\CC^{10}$.  The last result implies that the period map from the moduli space of Debarre-Voisin varieties to the relevant period space is birational. 
\end{abstract}

  \maketitle
\bibliographystyle{amsalpha}
\section{Introduction}\label{sec:intro}
\setcounter{equation}{0}
\subsection{Background and motivation}\label{subsec:retromotivi}
\setcounter{equation}{0}
The beautiful properties of vector bundles on $K3$ surfaces play a prominent r\^ole in algebraic geometry. Since $K3$ surfaces are the two-dimensional hyperk\"ahler (HK) compact manifolds, one is tempted   to  explore the world of vector bundles on higher dimensional HK's. In the present paper we give way to this temptation. 
Our proposal is to focus  attention on vector bundles, or more generally (coherent) torsion free sheaves, whose Chern character satisfies a certain condition, see Definition~\ref{effemod}. We call such sheaves modular. The definition is tailor made for  torsion-free sheaves  on  a polarized HK  $(X,h)$
 which deform to all small deformations of $(X,h)$. With this hypothesis we may   deform 
$(X,h)$ to a Lagrangian $(X_0,h_0,\pi)$, where $\pi\colon X_0\to\PP^n$ is a Lagrangian fibration, study stable sheaves $\cF$ on  
$X_0$ by  studying the restriction of $\cF$ to a generic fiber of $\pi$ (an abelian variety of dimension $n$), and 
then deduce properties of the initial moduli space of sheaves on $(X,h)$. This strategy was  implemented in the case of $K3$ surfaces, see~\cite{ogwt2}. A successful implementation in higher dimensions requires an extension  of the known results regarding the variation of $h$ slope-stability of a sheaf. More precisely one needs to know that, given a class $\chi\in H(X;\QQ)$, there exists a decomposition of the ample cone into open chambers such that, given a sheaf $\cF$ with $\ch(\cF)=\chi$ and an open chamber $\cC$, then $\cF$   is either $h$ slope-stable for all $h\in\cC$  or else for no such $h$. Modular sheaves on HK's are exactly the sheaves for which one can prove that such a decomposition of the ample cone exists. Another remarkable consequence of our definition is the following. Let $\pi\colon X\to\PP^n$ be a Lagrangian fibration, and let $\cF$ be a modular vector bundle on $X$ 
whose restriction  to a generic fiber of $\pi$ is slope-stable; then the restriction of $\cF$ to a generic fiber is a semi-homogeneous vector bundle (according to Mukai's definition, see~\cite{muksemi}), and hence it has no infinitesimal deformations fixing the determinant. This shows that, in the case of modular sheaves, the strategy outlined above in higher dimensions  resembles that which has been implemented in the case of $K3$ surfaces (notice that a slope-stable vector bundle on an elliptic curve  has by default no infinitesimal deformations fixing the determinant, while this  certainly does not hold for vector bundles on abelian surfaces - it holds exactly for semi-homogeneous ones). 
\subsection{Modular sheaves}\label{grandef}
\setcounter{equation}{0}
Let  $\cF$ be a rank $r$ torsion-free sheaf on a manifold $X$. The \emph{discriminant} $\Delta(\cF)\in H^{2,2}_{\ZZ}(X)$ is defined to be
\begin{equation}\label{dueversioni}
\Delta(\cF):=2r c_2(\cF)-(r-1) c_1(\cF)^2=-2r \ch_2(\cF)+\ch_1(\cF)^2.
\end{equation}
Below is our key definition. 
\begin{dfn}\label{effemod}
Let $X$ be a HK manifold of dimension $2n$, and let $q_X$ be its  Beauville-Bogomolov-Fuiki (BBF) bilinear symmetric form. A torsion free sheaf $\cF$ on $X$ is \emph{modular} if 
there exists $d(\cF)\in\QQ$ such that
\begin{equation}\label{fernand}
\int_X \Delta(\cF)\smile \alpha^{2n-2}=d(\cF) \cdot (2n-3)!! \cdot q_X(\alpha)^{n-1}
\end{equation}
   for all $\alpha\in H^2(X)$.
\end{dfn}
\begin{rmk}\label{semphodge}
Let $X$ be a HK variety of dimension $2n$. Let $D(X)\subset H(X)$  be the image of the map $\Sym H^2(X)\to H(X)$ defined by cup-product. Let $D^i(X):=D(X)\cap H^i(X)$. The  pairing $D^i(X)\times D^{4n-i}(X)\to\CC$ defined by intersection product is non degenerate~\cite{verbcohk,bogcohk,rhag}, hence there is a splitting $H(X)= D(X)\oplus D(X)^{\bot}$, where orthogonality is with respect to the intersection pairing. 
Now let  $\cF$ be a torsion free sheaf on $X$. Then $\cF$ is modular if and only if the orthogonal projection of $\Delta(\cF)$ onto $D^4(X)$ is a multiple of the class $q_X^{\vee}$ dual to $q_X$.
Moreover if  $\Delta(\cF)$ is a multiple of $c_2(X)$ then $\cF$ is modular by  Fujiki's formula, see Remark 4.12 in~\cite{fujiki}. 
\end{rmk}
\begin{rmk}\label{caroverb}
Let $X$ be a  HK of Type $K3^{[2]}$. Then $H(X)= D(X)$ (notation as in Remark~\ref{semphodge}). It follows that  a vector bundle $\cF$ on $X$  is modular if and only if  $\Delta(\cF)$ is a multiple of $c_2(X)$. It follows~\cite{verbhyper} that if $\cF$ is a modular vector bundle, slope-stable for a polarization $h$, then $End^0(\cF)$ is hyperholomorphic on $(X,h)$, where  $End^0(\cF)$ is  the vector bundle of traceless endomorphisms of $\cE$. More generally, on an arbitrary HK polarized variety $(X,h)$ there should be 
a relation between the property of being modular and that of being hyperholomorphic. 
\end{rmk}
\subsection{Main results}\label{risprin}
\setcounter{equation}{0}
Let  $\cE$ be a modular torsion-free sheaf on a hyperk\"ahler manifold $X$ of Type $K3^{[2]}$. A simple argument, see Proposition~\ref{restrango}, shows that $r(\cE)$ divides the square of a generator of the ideal $\{q_X(c_1(\cE),\alpha) \mid \alpha\in H^2(X;\ZZ)\}$. We  give an existence and uniqueness statement for  slope-stable vector bundles $\cE$ such that
 $r(\cE)$ \emph{equals}  the square of a generator of the ideal defined above, and moreover
\begin{equation*}
\Delta(\cE)  =  \frac{r(\cE)(r(\cE)-1)}{12}c_2(X).  
\end{equation*}
Before formulating our  results we recall the description of the irreducible components of the moduli space of polarized 
HK's  of Type $K3^{[2]}$. Let $(X,h)$ be one such polarized HK (we emphasize that the ample class $h\in H^{1,1}_{\ZZ}(X)$ is primitive). Then either
\begin{equation}\label{divuno}
q(h,H^2(X;\ZZ))=\ZZ,\quad q(h)=e>0, \quad e\equiv 0\pmod{2}
\end{equation}
or
\begin{equation}\label{divdue}
q(h,H^2(X;\ZZ))=2\ZZ,\quad q(h)=e>0,\quad e\equiv 6\pmod{8}.
\end{equation}
Conversely, if $e$ is a positive integer which is even (respectively congruent to $6$ modulo $8$) there exists $(X,h)$ 
such that~\eqref{divuno}  (respectively~\eqref{divdue}) holds. 
Let $\cK_{e}^1$ be the moduli space of polarized
 HK's $(X,h)$ of Type $K3^{[2]}$ such that~\eqref{divuno} holds, and let $\cK_{e}^2$ be the moduli space of polarized
 HK's $(X,h)$ of Type $K3^{[2]}$ such that~\eqref{divdue} holds. 
\begin{thm}\label{unicita}
Let $i\in\{1,2\}$ and let $r_0,e$ be positive integers such that  $r_0\equiv i\pmod{2}$ and
\begin{equation}\label{econ}
e\equiv
\begin{cases}
4r_0-10 \pmod{8r_0} & \text{if $r_0\equiv 0 \pmod{4}$,} \\
\frac{1}{2}(r_0-5) \pmod{2r_0} & \text{if $r_0\equiv 1 \pmod{4}$,} \\
-10 \pmod{8r_0}  & \text{if $r_0\equiv 2 \pmod{4}$,} \\
-\frac{1}{2}(r_0+5) \pmod{2r_0}  & \text{if $r_0\equiv 3 \pmod{4}$.}
\end{cases}
\end{equation}
Suppose that $[(X,h)]\in\cK^i_e$  is a generic point. Then up to isomorphism there exists one and only one $h$ slope-stable vector bundle $\cE$ on $X$ such that
 \begin{equation}\label{ele}
r(\cE)=r_0^2,\quad c_1(\cE)=\frac{r_0}{i} h,\quad \Delta(\cE)  =  \frac{r(\cE)(r(\cE)-1)}{12}c_2(X).
\end{equation}
Moreover $H^p(X,End^0(\cE))=0$ for all $p$. 
\end{thm}
\begin{rmk}
Let $[(X,h)]\in\cK^2_6$ be  generic. Then $(X,h)$ is isomorphic to the variety of lines $F(Y)$ on a generic cubic hypersurface  $Y\subset\PP^5$ polarized by the Pl\"ucker embedding, and  the vector bundle $\cE$  of Theorem~\ref{unicita} with $r_0=2$ is isomorphic to the restriction  of the tautological quotient vector bundle on $\Gr(2,\CC^6)$. Similarly, let $[(X,h)]\in\cK^2_{22}$ be generic. Then $(X,h)$ is isomorphic to the Debarre-Voisin variety 
associated to a generic $\sigma\in \bigwedge^3 V_{10}^\vee$, where 
  $V_{10}$ is a $10$ dimensional complex vector space, and 
\begin{equation}\label{eqDV}
X_\sigma:=\{[W]\in \Gr(6,V_{10}) \mid \sigma\vert_{ W}=0\}.
\end{equation}
 The vector bundle $\cE$  of Theorem~\ref{unicita} with $r_0=2$ is isomorphic to the restriction to $X_\sigma$ of the tautological quotient vector bundle on $\Gr(6,V_{10})$. These results are proved in Section~\ref{geodv}.
\end{rmk}
\begin{rmk}
We would like the congruence relations in~\eqref{econ} to be forced upon us by adding to the hypotheses on $r(\cE)$, $c_1(\cE)$, $\Delta(\cE)$ the extra hypothesis that $\chi(X,End^0(\cE))=0$. Our computations give this for some values of $r_0$, but we do not have complete results. 
\end{rmk}
The result below  replaces the  genericity hypothesis in Theorem~\ref{unicita} with a cohomological one.
\begin{crl}\label{solosolo}
Let $i\in\{1,2\}$ and let $r_0,e$ be positive integers such that  $r_0\equiv i\pmod{2}$ and~\eqref{econ} holds.   
Let $[(X,h)]\in\cK^i_e$. Suppose  
 that $\cE$ is an $h$ slope-stable vector bundle  on $X$ such  that~\eqref{ele} holds and $H^2(X,End^0(\cE))=0$. If  $\cG$ is an  $h$ slope-stable vector bundle on $X$ and  $\ch_k(\cG)=\ch_k(\cE)$ for $k\in\{0,1,2\}$ then $\cG$ is isomorphic to $\cE$. 
\end{crl}
There is an interesting consequence of Theorem~\ref{unicita} involving Debarre-Voisin varieties.
There is a GIT moduli space 
$\cM_{DV}:=\PP(\bigwedge^3 V_{10}^\vee)\gquot \SL(V_{10})$ of Debarre-Voisin varieties, see~\cite{dhov-journal}. In~\cite{devo} it was proved that the moduli map 
$\cM_{DV}\dra\cK^2_{22}$ has finite \emph{non zero} degree.
\begin{thm}\label{perdivi}
The moduli map $\cM_{DV}\dra\cK^2_{22}$ is birational.
\end{thm}
Theorem~\ref{unicita} is also relevant to the study of  degenerate DV varieties which was carried out in~\cite{dhov-journal}. We say a few words about this in Subsection~\ref{futuro}.  

\subsection{Outline of the paper}
\setcounter{equation}{0}
In Section~\ref{naturale} we give a few examples of modular sheaves, and we make the connection with semi-homogeneous vector bundles. In particular we give strong restrictions on the possible ranks of modular sheaves, under some hypotheses.

Section~\ref{camere} contains the results that extend to modular sheaves the known results on the variation of slope-stability for sheaves on surfaces. In particular we show that one can extend to HK's with a Lagrangian fibration the results that hold for  sheaves on surfaces which are fibered over a curve.

In Section~\ref{stablag} we prove properties of slope-stable modular vector bundles on HK's $X$ of Type $K3^{[2]}$ with a Lagrangian fibration $X\to\PP^2$. We make certain hypotheses, in particular we assume that $r(\cE)$ equals the square of a generator of the ideal 
$\{q_X(c_1(\cE),\alpha) \mid \alpha\in H^2(X;\ZZ)\}$. We show that the restriction of a slope-stable modular vector bundle on $X$  to a generic Lagrangian fiber is slope-stable, and that if $\cF$ is another such vector bundle then  the restrictions of $\cE$ and
 $\cF$ to a generic Lagrangian fiber are isomorphic. 

Section~\ref{esempi}  discusses a construction which associates to a vector bundle $\cF$ on a $K3$ surface $S$ two torsion free sheaves $\cF[n]^{\pm}$ on $S^{[n]}$ whose fibers over a reduced scheme $\{x_1,\ldots,x_n\}$ are the tensor product $\cF(x_1)\otimes\ldots\otimes\cF(x_n)$ of the fibers of $\cF$ at the points $x_1,\ldots,x_n$ - this is a generalization of a construction which was given in~\cite{dhov-journal}. We prove that if $\chi(S,\cF^{\vee}\otimes\cF)=2$, then $\cF[2]^{\pm}$ is a modular vector bundle, and we compute its Chern character. As proved in Section~\ref{ledimo} this construction gives (by deformation) the existence result of Theorem~\ref{unicita}. 
 
In Section~\ref{ambrose} we let $S\to\PP^1$ be an elliptic $K3$ surface with Picard number $2$. Then $S^{[2]}$ has an associated Lagrangian fibration $\pi\colon S^{[2]}\to(\PP^1)^{(2)}\cong\PP^2$. We prove that if  $\cF$ is a slope-stable rigid vector bundle on $S$ then  the vector bundle $\cF[2]^{\pm}$ on $S^{[2]}$ has  good properties. In particular we show that it extends to any small deformation of $S^{[2]}$ which keeps  $c_1(\cF[2]^{\pm})$ of type $(1,1)$, and that the restriction to any fiber of the Lagrangian fibration $\pi$ is simple.

Section~\ref{ledimo} contains the proof of Theorem~\ref{unicita} (and of Corollary~\ref{solosolo}). The basic idea is as follows. 
Let $\cX\to T^i_e$ be a complete family of polarized HK's of Type $K3^{[2]}$ whose moduli belong to $\cK^i_e$. 
By Gieseker and Maruyama there is a relative moduli scheme $f\colon \cM_e(r_0)  \to T^i_e$ whose fiber over $t\in T^i_e$  is  the moduli space of slope-stable vector bundles on $(X_t,h_t)$ with the given rank, $c_1$ and $c_2$. The map  $f\colon \cM_e(r_0)  \to T^i_e$ is of finite type by a result of Maruyama. Let $ \cM_e^{*}(r_0)\subset  \cM_e(r_0)$ be the (open) subset parametrizing vector bundles whose $\ch_3$ and $\ch_4$ is given by the formulae in Theorem~\ref{unicita}.
Because of the good properties of the vector bundles $\cF[2]^{\pm}$, the image  $f( \cM_e^{*}(r_0))$ contains a  dense open (in the Zariski topology) subset of $T^i_e$. On the other hand the results of Section~\ref{stablag} and~\ref{ambrose} allow us to prove that, up to isomorphism, there is a unique slope-stable vector bundle with the relevant $\ch_0,\ch_1,\ch_2$ on a generic HK parametrized by a Lagrangian Noether-Lefschetz locus with large discriminant. By density of the union of Noether-Lefschetz divisors (with large discriminant) we conclude that 
$f$ has degree $1$.

In Section~\ref{geodv} we prove Theorem~\ref{perdivi}. Once we have Theorem~\ref{unicita}, the main point is to show that the tautological quotient vector bundle on a generic DV variety is slope-stable.

In the appendix we discuss  properties of semi-homogeneous vector bundles on abelian varieties, and of Lagrangian Noether-Lefschetz divisors on moduli spaces of polarized HK's of Type $K3^{[2]}$. 
\subsection{Conventions}\label{subsec:ginevra}
\setcounter{equation}{0}
\begin{enumerate}
\item[$\bullet$] 
Algebraic variety is synonymous of complex quasi projective variety (not necessarily irreducible). 
\item[$\bullet$] 
Let $X$ be a smooth complex quasi projective variety and $\cF$ a  coherent sheaf  on $X$. We only consider  topological  Chern classes   $c_i(\cF)\in H^{2i}(X(\CC);\ZZ)$.
\item[$\bullet$] 
Let $X$ be a HK manifold of dimension $2n$. We let $q_X$, or simply $q$, be the BBF symmetric bilinear form of $X$, and we denote $q_X(\alpha,\alpha)$ by $q_X(\alpha)$. We let $c_X$ be the \emph{normalized Fujiki constant of $X$}, i.e.~the rational positive number such that for all $\alpha\in H^2(X)$ we have
\begin{equation}
\int\limits_{X}\alpha^{2n}=c_X\cdot (2n-1)!!\cdot q_X(\alpha)^n.
\end{equation}
A hyperk\"ahler (HK) \emph{variety} is a projective compact HK manifold. 
\item[$\bullet$] 
Let $\cF$ be a torsion-free sheaf on a polarized projective variety $(X,h)$. A  subsheaf  $\cE\subset\cF$ is \emph{slope-destabilizing} if $0<r(\cE)<r(\cF)$ and $\mu_h(\cE)\ge \mu_h(\cF)$, where $r(\cE),r(\cF)$ are the ranks of $\cE,\cF$, and $\mu_h(\cE),\mu_h(\cF)$ are the $h$-slopes of $\cE,\cF$.   If $\mu_h(\cE)> \mu_h(\cF)$ then 
$\cE\subset\cF$ is \emph{slope-desemistabilizing}. We use similar terminology for exact sequences $0\to\cE\to\cF\to \cG\to 0$. 
\item[$\bullet$] 
A torsion-free sheaf  on  $(X,h)$ is \emph{strictly} $h$ slope-semistable if it is $h$ slope-semistable but not $h$ slope-stable. 
\item[$\bullet$] 
Abusing notation we say that a smooth projective variety $X$ is an  abelian variety if it is isomorphic to the variety underlying an abelian variety $A$. In other words 
$X$ is a torsor of $A$. 
\end{enumerate}
\subsection{Acknowledgements}
\setcounter{equation}{0}
Many thanks go to the (anonymous) referee for his numerous suggestions. In particular I owe to the referee the proof of Proposition~\ref{prp:seirigido} via the MacKay correspondence (the original proof was done by brute force)
\section{Modular  sheaves}\label{naturale}
\setcounter{equation}{0}
\subsection{First examples}\label{trecasi}
\setcounter{equation}{0}
By Remark~\ref{semphodge} the following are  modular vector bundles:
\begin{enumerate}
\item
The tangent bundle $\Theta_X$. 
\item
Let $V_6$ be a $6$ dimensional complex vector space, and let $X\subset\Gr(2,V_6)$ be the variety of lines contained in a smooth cubic hypersurface in $\PP(V_6)$. Let
 $h\in H^{1,1}_{\ZZ}(X)$ be the Pl\"ucker polarization. Then $X$ is a HK of type $K3^{[2]}$, see~\cite{beaudon}. Let $\cQ$  be  the restriction to $X$ of the tautological rank $4$  quotient vector bundle on $\Gr(2,V_6)$. We claim that    
\begin{equation}\label{yaris}
\ch_0(\cQ)  =  4, \quad \ch_1(\cQ)  =  h, \quad \ch_2(\cQ)  =  \frac{1}{8}\left(h^2 -c_2(X)\right).
\end{equation}
 The first two equations are obvious, the last equation can be obtained as follows. Let $\cU$ be the restriction to $X$ of the tautological subbundle on $\Gr(2,V_6)$.  
The normal bundle sequence
 \begin{equation*}
0\lra \Theta_X \lra \Theta_{\Gr(2,V_6)|X} \lra \Sym^3 \cU^{\vee}\lra 0
\end{equation*}
gives that  $\ch_2(\cU)  = -\left(h^2 -c_2(X)\right)/8$. Since $\ch(\cQ)=6-\ch(\cU)$ this gives the last equation in~\eqref{yaris}.
 Thus $\Delta(\cQ)=c_2(X)$, and hence $\cQ$ is modular.  
\item
Let $X\subset \Gr(6,V_{10})$ be a smooth DV variety  and  let 
 $h\in H^{1,1}_{\ZZ}(X)$ be the  Pl\"ucker polarization, see~\cite{devo}.  Then $X$ is a HK of Type $K3^{[2]}$. Let 
$\cQ$ be the  the restriction to $X$ of the tautological rank $4$  quotient vector bundle on $\Gr(6,V_{10})$. Then   
\begin{equation}
\ch_0(\cQ)  =  4, \quad \ch_1(\cQ)  =  h, \quad \ch_2(\cQ)  =  \frac{1}{8}\left(h^2 -c_2(X)\right).
\end{equation}
The above equations follow from the computations on p.~83 of~\cite{devo}, see Lemma~\ref{classidv}. Thus $\Delta(\cQ)=c_2(X)$  and hence $\cQ$ is modular. 
\end{enumerate}
\begin{rmk}\label{primipassi}
Let $X$ be a HK variety and let $\cE,\cF$ be  modular sheaves on $X$. Then  $\cE\oplus\cF$ is not  modular in general. On the other hand $\cE\otimes\cF$ is  modular, at least if $\cE$ and $\cF$ are locally free.
\end{rmk}
\begin{rmk}\label{deltanti}
Let $X$ be a HK manifold of dimension $2n$, and let $\cF$ be a torsion free  modular sheaf on $X$.  
Then 
\begin{equation}
\int\limits_{X}\Delta(\cF)\smile\alpha_1\smile\ldots\smile \alpha_{2n-2}=d(\cF)\cdot\wt{\sum}\, q_X(\alpha_{i_1},\alpha_{i_2})\cdot\ldots\cdot q_X(\alpha_{i_{2n-3}},\alpha_{i_{2n-2}}),
\end{equation}
for all $\alpha_1,\ldots, \alpha_{2n}\in H^2(X)$, where $\wt{\sum}$ means that in the summation we avoid repeating addends which are formally equal (i.e.~are equal modulo reordering of the factors   $q_X(\cdot,\cdot)$'s and switching the entries in $q_X(\cdot,\cdot)$). 
\end{rmk}
\subsection{Restrictions on the rank}
\setcounter{equation}{0}
Below is the result that was mentioned in Subsection~\ref{risprin}.
\begin{prp}\label{restrango}
Let $X$ be a HK fourfold of Type $K3^{[2]}$ or $\Kum_2$. Let $\cF$ be a modular torsion-free sheaf on $X$. Let $m$ be a generator of the ideal 
\begin{equation*}
\{q_X(c_1(\cF),\alpha)\mid \alpha\in H^2(X;\ZZ)\}.
\end{equation*}
 Then  $r(\cF)$ divides $m^2$ if $X$ is of Type $K3^{[2]}$, and it divides $3m^2$ if $X$ is of Type $\Kum_2$. 
\end{prp}
\begin{proof}
As is easily checked, there exists $\alpha\in H^2(X;\ZZ)$  such that  $q_X(c_1(\cF),\alpha)=m$ and $q_X(\alpha)=0$. Let $r:=r(\cF)$. Since $q_X(\alpha)=0$, Equation~\eqref{fernand} gives that
\begin{multline}
2r\int_X c_2(\cF)\smile\alpha^2=(r-1)\int_X c_1(\cF)^2\smile\alpha^2= \\
=2(r-1) c_X \cdot q_X(c_1(\cF),\alpha)^2=2(r-1) c_X\cdot m^2.
\end{multline}
The result follows because  $c_X=1$ if $X$ is of Type $K3^{[2]}$, and $c_X=3$ if $X$ is of Type $\Kum_2$.
\end{proof}
\subsection{Modular sheaves on Lagrangian fibrations}
\setcounter{equation}{0}
We recall that  a Lagrangian fibration $\pi\colon X\to\PP^n$ on a HK manifold $X$ of dimension $2n$ is  a surjective map with connected fibers whose smooth fibers  are  abelian varietes.
\begin{rmk}
For $t\in\PP^n$ we let $X_t:=\pi^{-1}(t)$ be the schematic fiber over $t$. If $X_t$ is smooth the image of the restriction map $H^2(X;\ZZ)\to H^2(X_t;\ZZ)$ has rank one, and is generated by  an ample 
class $\theta_t\in H^{1,1}_{\ZZ}(X_t)$, see~\cite{wieneck1}.  If $\cF$ is a sheaf on $X_t$ slope-(semi)stability of $\cF$ will always mean $\theta_t$ slope-(semi)stability.
\end{rmk}
If  $\pi\colon X\to\PP^n$ is  a Lagrangian fibration we let 
 \begin{equation}\label{eccoeffe}
f:=c_1(\pi^{*}\cO_{\PP^n}(1))\in H^{1,1}_{\ZZ}(X).
\end{equation}
As is well-known $q_X(f)=0$.
\begin{lmm}\label{zerene}
Let  $\pi\colon X\to\PP^n$ be a  Lagrangian fibration of a HK manifold of dimension $2n$. Suppose that $\cF$ is a modular torsion free sheaf on $X$. Let $t\in\PP^n$ be a general point, and let $\cF_t:=\cF_{|X_t}$ be the restriction of $\cF$ to $X_t$. Then 
\begin{equation}\label{intzero}
\int_{X_t}\Delta(\cF_t)\smile \theta_t^{n-2}=0.
\end{equation}
\end{lmm}
\begin{proof}
There exists $\omega\in H^2(X;\ZZ)$ such that $\theta_t=\omega_{|X_t}$. Since $t\in\PP^n$ is a generic point, we have $\Delta(\cF_t)=\Delta(\cF)_{|X_t}$. Moreover $f^n$ is the Poincar\'e dual of $X_t$.  Hence
\begin{equation}
\int_{X_t}\Delta(\cF_t)\smile \theta_t^{n-2}=\int_{X}\Delta(\cF)\smile \omega^{n-2}\smile f^n.
\end{equation}
The integral on the right vanishes by Remark~\ref{deltanti} and the equality
 $q_X(f)=0$.
\end{proof}
\begin{expl}
Let $S$ be a $K3$ surface, and let $V$ be a vector bundle on $S$. Let $\cZ\subset S\times S^{[n]}$ be the tautological subscheme, and let $p\colon\cZ\to S$, $q\colon \cZ\to S^{[n]}$ be the projection maps. The locally free sheaf  $q_{*}(p^{*}V)$ is known as a \emph{tautological sheaf} on $S^{[n]}$. In general such a sheaf is not modular. In fact suppose that $S$ is elliptic, with elliptic fibration $S\to\PP^1$. The composition  $S^{[n]}\to S^{(n)}\to  (\PP^1)^{(n)}\cong\PP^n$ is a Lagrangian fibration with generic fiber $X_t=C_1\times\ldots\times C_n$, where $C_1,\ldots, C_n$ are generic distinct fibers of the elliptic fibration $S\to\PP^1$.  If the restriction of $V$ to the fibers of $S\to\PP^1$ has non zero degree then  Equality~\eqref{intzero} does not hold for $\cF:=q_{*}(p^{*}V)$, and hence $q_{*}(p^{*}V)$ is not modular. 
\end{expl}
\begin{prp}\label{resemi}
Let  $\pi\colon X\to\PP^n$ be a   Lagrangian fibration of a HK manifold of dimension $2n$. Let $\cF$ be a modular torsion free sheaf on $X$. 
Suppose that $t\in\PP^n$ is a regular value of $\pi$, that $\cF$ is locally-free in a neighborhood of $X_t$, and that $\cF_t$ is slope-stable. Then $\cF_t$ is a semi homogeneous vector bundle.
\end{prp}
\begin{proof}
Follows from Lemma~\ref{zerene} and Proposition~\ref{discperp}. 
\end{proof}
The  result below shows that, under suitable hypotheses, a much stronger version of Proposition~\ref{restrango} holds.
\begin{crl}\label{resemibis}
Let $X$ be a HK of Type $K3^{[n]}$, $\Kum_n$ or OG6. Let $\cF$ be a modular torsion free sheaf on $X$. 
Suppose that $t\in\PP^n$ is a regular value of $\pi$, that $\cF$ is locally-free in a neighborhood of $X_t$, and that $\cF_t$ is slope-stable. Then there exist  positive integers $r_0,d$, with $d$ dividing $c_X$,   such that   
$ r(\cF)= \frac{r_0^{n}}{d}$. 
\end{crl}
\begin{proof}
If $X$ is of Type $K3^{[n]}$ then $c_X=1$ and $\theta_t$ is a principal polarization, see~\cite{wieneck1}. If $X$ is of Type $\Kum_n$ or OG6 then $c_X=n+1$ and $\theta_t$ is a polarization with elementary divisors $(1,\ldots,1,d_1,d_2)$ where $d_1\cdot d_2$ divides $n+1$ see~\cite{wieneck2} for $\Kum_n$ and~\cite{monrap-monog6} for OG6. 
Hence the result  follows from  Proposition~\ref{resemi} and Proposition~\ref{potenza}.
\end{proof}
\section{Variation of stability for modular  sheaves}\label{camere}
\setcounter{equation}{0}
\subsection{Main results}
\setcounter{equation}{0}
Let $X$ be an irreducible smooth projective variety. If the ample cone $\Amp(X)$ has rank greater than $1$ (and hence $\dim X\ge 2$), slope-stability of a sheaf $\cF$ depends on the choice of an ample ray. If $X$ is a surface there is a locally finite decomposition  $\Amp(X)_{\RR}$ into chambers defined by rational walls  such that slope-stability is the same for ample classes belonging to the same open chamber. One important feature is that the walls 
 depend only on the Chern character of $\cF$.

If $\dim X\ge 3$  the picture is  more intricate in general, see for example~\cite{grebmaster}.

In the present section we  show that if $X$ is a HK variety and $\cF$ is a modular sheaf, then there is a decomposition of $\Amp(X)_{\RR}$ as if $X$  were a surface. 
\begin{dfn}
Let $a$ be a positive real number. An \emph{$a$-wall} of $\Amp(X)_{\RR}$ is the intersection 
$\lambda^{\bot}\cap \Amp(X)_{\RR}$, where 
  $\lambda\in H^{1,1}_{\ZZ}(X)$, 
$ -a \le q_X(\lambda)< 0$,
and orthogonality is with respect to the BBF quadratic form $q_X$.
\end{dfn}
As is well-known, the set of $a$-walls  is  locally finite, in particular the union of all the $a$-walls  is closed in $\Amp(X)_{\RR}$. 
\begin{dfn}
An \emph{open $a$-chamber} is a connected component of the complement (in $\Amp(X)_{\RR}$) of   the union of all the $a$-walls. 
\end{dfn}
\begin{dfn}\label{adieffe}
Let $X$ be a HK manifold, and let $\cF$ be a modular torsion free sheaf  on $X$.
Then
\begin{equation}\label{esmeralda}
a(\cF):=\frac{r(\cF)^2 \cdot d(\cF) }{4c_X},
\end{equation}
where $d(\cF)$ is as in Definition~\ref{effemod}.
\end{dfn}
Below is the first main result.
\begin{prp}\label{campol}
Let $X$ be a  HK variety of dimension $2n$, and let  $\cF$ be a   torsion free modular sheaf  on $X$. Then the following hold:
\begin{enumerate}
\item
Suppose that $h$ is an ample divisor class on $X$ which belongs to an open $a(\cF)$-chamber. If  $\cF$ is  strictly $h$ slope-semistable  there exists an exact sequence
of torsion free non zero sheaves
\begin{equation}
0\lra \cE\lra \cF\lra \cG\lra 0
\end{equation}
such that $r(\cF) c_1(\cE)-r(\cE) c_1(\cF)=0$.
\item
Suppose that $h_0,h_1$ are ample divisor classes on $X$  belonging to the same open $a(\cF)$-chamber. Then $\cF$ is 
 $h_0$ slope-stable if and only if it is $h_1$ slope-stable. 
\end{enumerate}
\end{prp}
Proposition~\ref{campol} is proved in Subsection~\ref{camerehk}. 

The next result is about slope-stable sheaves on HK varieties which carry a Lagrangian fibration.  
\begin{dfn}\label{suipol}
Let $X$ be a HK variety equipped with a Lagrangian fibration $\pi\colon X\to\PP^n$, and  let $f:=\pi^{*}c_1(\cO_{\PP^n}(1))$. Let $a$ be positive integer. An ample divisor class $h$ on $X$ is 
$a$-\emph{suitable}  if the following holds. Let $\lambda\in H^{1,1}_{\ZZ}(X)$ be a class such that $-a\le q_X(\lambda)< 0$: then either $q_X(\lambda,h)$ and $q_X(\lambda,f)$ have the same sign, or they are both zero. 
\end{dfn}
Notice that the notion of $a$-suitable depends on the chosen Lagrangian fibration. 
\begin{prp}\label{lagstab}
Let $\pi\colon X\to\PP^n$ be a  Lagrangian fibration of a HK variety  
of dimension $2n$. Let $\cF$ be  a torsion free modular sheaf   on $X$ such that $\sing \cF$ does \emph{not} dominate $\PP^n$. Let $h$ be an ample divisor class on $X$ which is $a(\cF)$-suitable.
   Then the following hold:
\begin{enumerate}
\item[(i)]
If the restriction of $\cF$ to a generic fiber of $\pi$   is  slope-stable,  then $\cF$ is $h$ slope-stable. 
\item[(ii)]
If $\cF$ is $h$ slope-stable then the restriction of $\cF$ to the generic fiber of $\pi$   is  slope-semistable.
\end{enumerate}
\end{prp}
Proposition~\ref{lagstab} is proved in Subsection~\ref{pazzia}.

\subsection{Change of slope-stability and strictly semistable sheaves}
\setcounter{equation}{0}
Suppose that
$\cE,\cF$  are sheaves on an irreducible smooth  variety $X$. We let
\begin{equation}\label{pecora}
\lambda_{\cE,\cF}:=(r(\cF) c_1(\cE)-r(\cE) c_1(\cF))\in H^2(X;\ZZ).
\end{equation}
\begin{lmm}\label{comesup}
Let $(X,h)$ be a polarized HK variety, and let $\cE,\cF$  be non zero torsion free sheaves on  $X$.
Then
\begin{enumerate}
\item[(a)]
$\mu_h(\cE)>\mu_h(\cF)$ if and only if $q_X(\lambda_{\cE,\cF},h)>0$.
\item[(b)]
$\mu_h(\cE)=\mu_h(\cF)$ if and only if $q_X(\lambda_{\cE,\cF},h)=0$.
\end{enumerate}
\end{lmm}
\begin{proof}
Let  $2n$  be the dimension of $X$. We have $\mu_h(\cE)>\mu_h(\cF)$ if and only if 
$\int_X\lambda_{\cE,\cF}\smile h^{2n-1}>0$, and  by Fujiki's formula this holds if and only if 
\begin{equation*}
c_X\cdot(2n-1)!!\cdot q_X(\lambda_{\cE,\cF},h)\cdot q_X(h)^{n-1}>0.
\end{equation*}
Item~(a) follows, because $c_X>0$ and $q_X(h)>0$. 

 We have $\mu_h(\cE)=\mu_h(\cF)$ if and only if 
$\int_X\lambda_{\cE,\cF}\smile h^{2n-1}=0$, and hence   Item~(b) follows again by Fujiki's formula.
\end{proof}
\begin{prp}\label{polint}
Let $X$ be a HK variety, and let $h_0,h_1$ be ample divisor classes on $X$. Suppose that $\cF$ is a torsion free sheaf on $X$ which is $h_0$ slope-stable and not $h_1$ slope-stable. Then there exists $h\in(\QQ_{+}h_0+\QQ_{+}h_1)$ such that $\cF$ is strictly $ h$ slope-semistable, i.e.~$\cF$ is $h$ slope-semistable but not $h$ slope-stable. 
\end{prp}
\begin{proof}
Lemma~\ref{comesup} allows to reproduce the  proof of the analogous statement valid for  surfaces (see~\cite{huylehnbook}). Let ${\mathsf S}\subset ([0,1]\cap \QQ)$ be the set of $s$ for which there exists a subsheaf  $\cE\subset\cF$ 
with $0<r(\cE)<r(\cF)$ such that
\begin{equation}\label{segretarie}
q_X(\lambda_{\cE,\cF},(1-s)h_0+s h_1)= 0.
\end{equation}
Then ${\mathsf S}$ is non empty and finite. 
In fact by hypothesis there exists    an $h_1$ destabilizing subsheaf $\cE\subset\cF$. Thus
 $0<r(\cE)<r(\cF)$, and 
 $q_X(\lambda_{\cE,\cF}, h_1)\ge 0$ by Lemma~\ref{comesup}. On the other hand, by the same lemma, $q_X(\lambda_{\cE,\cF}, h_0)< 0$ 
because $\cF$ is $h_0$ slope-stable. It follows that there exists $s\in[0,1]\cap \QQ$ such that~\eqref{segretarie} holds,
 i.e.~$\mathsf S$ is not empty. 
 
 In order to prove that $\mathsf S$ is finite, assume that~\eqref{segretarie} holds.
Since $\cF$ is $h_0$ slope-stable,  $q_X(\lambda_{\cE,\cF}, h_0)< 0$. By linearity of $q_X(\lambda_{\cE,\cF},\cdot)$ we get that   
$q_X(\lambda_{\cE,\cF},h_1)\ge 0$. By Lemma~\ref{comesup} it follows that
\begin{equation}\label{dasotto}
\mu_{h_1}(\cE)\ge \mu_{h_1}(\cF).
\end{equation}
The set of subsheaves  $\cE\subset\cF$ such that~\eqref{dasotto} holds is bounded (see Lemma 1.7.9 in~\cite{huylehnbook}), i.e.~up to isomorphism each such sheaf belongs to a finite set of families, each parametrized by an irreducible quasi projective variety. It follows that $\mathsf S$ is finite because the values of $q_X(\lambda_{\cE,\cF}, h_i)$ for $i\in\{0,1\}$ are constant for sheaves $\cE$ parametrized by an irreducible variety.

Since $\mathsf S$ is finite, there is  a minimum $s$, call it $s_{\min}$, such that~\eqref{segretarie} holds for some 
 subsheaf  $\cE\subset\cF$ with $0<r(\cE)<r(\cF)$.
Clearly  $\cF$ is strictly $(h_0+s_{\min} h_1)$ slope-semistable.
\end{proof}
\subsection{Strictly semistable modular sheaves}\label{camerehk}
\setcounter{equation}{0}
\begin{lmm}\label{mercedes}
Let
\begin{equation}\label{effegi}
0\lra \cE\lra \cF\lra \cG\lra 0
\end{equation}
be an   exact sequence of  sheaves on a smooth variety.  Then
\begin{equation}\label{lungaeq}
r(\cF) \cdot r(\cG) \Delta(\cE)+ r(\cF) \cdot r(\cE) \Delta(\cG)=r(\cE)\cdot r(\cG)\Delta(\cF)+\lambda_{\cE,\cF}^2.
\end{equation}
\end{lmm}
\begin{proof}
Follows from additivity of the Chern character, and the second equality in~\eqref{dueversioni}. 
\end{proof}
\begin{prp}\label{propsemi}
Let $(X,h)$ be a polarized HK variety of dimension $2n$.   
Let $\cF$ be  a    torsion free modular \emph{strictly}  $h$ slope-semistable  sheaf on $X$,  and let 
\begin{equation}\label{faria}
0\lra \cE\lra \cF\lra \cG\lra 0
\end{equation}
be an   exact sequence of non zero torsion free sheaves which is $h$ slope destabilizing, i.e.~$\mu_h(\cE)=\mu_h(\cF)$. Then
\begin{equation}\label{doppelgang}
-a(\cF) \le q_X(\lambda_{\cE,\cF})\le 0.
\end{equation}
Moreover $q_X(\lambda_{\cE,\cF})= 0$ only if $\lambda_{\cE,\cF}=0$.  
\end{prp}
\begin{proof}
Since the exact sequence in~\eqref{faria} is destabilizing,  
$q_X(\lambda_{\cE,\cF}, h)=0$ by Lemma~\ref{comesup}. Since the BBF form on $\NS(X)$ has signature $(1,\rho(X)-1)$, it follows that 
$q_X(\lambda_{\cE,\cF})\le 0$ with equality only if $\lambda_{\cE,\cF}=0$.  (Recall that $q_X(h)> 0$, because $h$ is ample.)

We are left with proving the second inequality in~\eqref{doppelgang}. Hence we assume that $ q_X(\lambda_{\cE,\cF})< 0$. 
 Cupping both sides of the equality in~\eqref{lungaeq} by $h^{2n-2}$, and integrating, we get (here we use the hypothesis that  
$\cF$ is modular)
\begin{equation}\label{pharaon}
\begin{aligned}
\int_X r(\cF) \cdot r(\cG) \Delta(\cE)\smile h^{2n-2}+\int_X  r(\cF) \cdot r(\cE) \Delta(\cG)\smile h^{2n-2}=\\
=r(\cE)\cdot r(\cG)\cdot d(\cF)\cdot(2n-3)!! q_X(h)^{n-1} +c_X\cdot q_X(\lambda_{\cE,\cF})\cdot (2n-3)!! q_X(h)^{n-1}. \\
\end{aligned}
\end{equation}
By hypothesis $\mu_h(\cE)=\mu_h(\cF)=\mu_h(\cG)$. Since  $\cF$ is $h$ slope-semistable it follows that $\cE$ and $\cG$ are $h$ slope-semistable torsion free sheaves. Thus 
$$\int_X \Delta(\cE)\smile h^{2n-2}\ge0 ,\quad \int_X \Delta(\cG)\smile h^{2n-2}\ge0 $$
 by Bogomolov's inequality, and hence~\eqref{pharaon} gives
\begin{equation}\label{carconte}
-r(\cE)\cdot r(\cG)\cdot d(\cF) \le c_X\cdot q_X(\lambda_{\cE,\cF}).
\end{equation}
Dividing by $c_X$ (which is strictly positive), we see that the second inequality in~\eqref{doppelgang} follows from~\eqref{carconte} and the inequality $r(\cE)\cdot r(\cG)\le r(\cF)^2/4$. 
\end{proof}

\subsection{Proof of Proposition~\ref{campol}}\label{dimcampol}
\setcounter{equation}{0}
Item~(1) follows  from Proposition~\ref{propsemi}. We prove Item~(2). By symmetry, it suffices to show that if $\cF$ is $h_0$ slope-stable, then it is $h_1$ slope-stable. Suppose that $\cF$ is not $h_1$ slope-stable. By Proposition~\ref{polint}, there exists $h\in(\QQ_{+}h_0+\QQ_{+}h_1)$ such that $\cF$ is strictly $ h$ slope-semistable. Hence there exists an $h$ destabilizing
\begin{equation*}
0\lra \cE\lra \cF\lra \cG\lra 0
\end{equation*}
   exact sequence of non zero torsion free  sheaves.  Since $h_0,h_1$ belong to the same open $a(\cF)$ chamber, also $h$ belongs to the same open $a(\cF)$-chamber. Thus, by Proposition~\ref{propsemi},  we get that $\lambda_{\cE,\cF}=0$. It follows that $\cF$ is not $h_0$ slope-stable, and that is a contradiction.
\qed
\subsection{Stability of modular sheaves on a lagrangian HK}\label{pazzia}
\setcounter{equation}{0}
\begin{lmm}\label{contreffe}
Let $X$ be a HK variety  of dimension $2n$ 
equipped with a Lagrangian fibration $\pi\colon X\to\PP^n$, and let $f:=c_1(\pi^{*}\cO_{\PP^n}(1))$. 
Let $\cF$ be a torsion free sheaf on $X$, and let $\cE\subset\cF$ be a  subsheaf with $0<r(\cE)<r(\cF)$. Then the following hold:
\begin{enumerate}
\item[(a)]
If, for generic $t\in\PP^n$, the restriction $\cF_t:=\cF_{|_{X_t}}$ is slope-stable, then
\begin{equation}\label{intneg}
q_X(\lambda_{\cE,\cF},f)<0.
\end{equation}
\item[(b)]
If, for generic $t\in\PP^n$, the subsheaf $\cE_t:=\cE_{|_{X_t}}\subset\cF_t$ is  slope  desemistabilizing, then
\begin{equation}\label{intpos}
q_X(\lambda_{\cE,\cF},f)> 0.
\end{equation}
\end{enumerate}
\end{lmm}
\begin{proof}
Let $h_t:=h_{|X_t}$. We have
\begin{equation}\label{fujidache}
\begin{aligned}
\int_{X_t} \lambda_{\cE_t,\cF_t}\smile h_t^{n-1}=\int_X \lambda_{\cE,\cF}\smile h^{n-1}\smile f^n= \\
=n! c_X \cdot q_X(h,f)^{n-1} \cdot q_X(\lambda_{\cE,\cF},f).
\end{aligned}
\end{equation}
In fact, the first equality holds because  $f^n$ is the Poincar\`e dual of $X_t$, and the second equality holds by Fujiki's formula and the fact that 
 $q_X(f)=0$. Items~(a) and~(b) follow, because $c_X$ and $q_X(\cO_X(h),f)$ are strictly positive.
\end{proof}

\begin{proof}[Proof of Proposition~\ref{lagstab}]
We prove Item~(i). Suppose that $\cF$ is not $h$ slope-stable. Let $\mathsf S\subset([0,1]\cap\QQ)$ be the set  of $s$ for which there exists a subsheaf $\cE\subset\cF$,  with $0<r(\cE)<r(\cF)$,  such that 
\begin{equation}\label{comenovem}
q_X(\lambda_{\cE,\cF},(1-s)h+s f)=0.
\end{equation}
Let us show that $\mathsf S$ is non empty and  finite.  Since $\cF$  is not $h$ slope-stable,  by  Lemma~\ref{comesup}  there exists  a  subsheaf 
$\cE\subset\cF$,  with $0<r(\cE)<r(\cF)$,  such that $q_X(\lambda_{\cE,\cF},h)\ge 0$. On the other hand, by  Lemma~\ref{contreffe}, the inequality in~\eqref{intneg} holds. It follows that $\mathsf S$ is not empty. The argument, in the proof of Proposition~\ref{polint}, showing that the analogous $\mathsf S$ is finite, applies also in the present case, and hence  $\mathsf S$ is finite.

Let  $s_{\min}$ be the minimum element of  $\mathsf S$. Clearly $\cF$ is strictly $h+s_{\min} f$ slope-semistable. Let $\cE\subset\cF$ be a subsheaf, with $0<r(\cE)<r(\cF)$ which is $h+s_{\min} f$ destabilizing, i.e.~$q_X(\lambda_{\cE,\cF},h+s_{\min} f)= 0$. Then $-a(\cF)\le q_X(\lambda_{\cE,\cF})\le 0$ by Proposition~\ref{propsemi}. On the other hand, $q_X(\lambda_{\cE,\cF}, f)< 0$ by Lemma~\ref{contreffe}, and hence $q_X(\lambda_{\cE,\cF}, h)< 0$ by our hypothesis on $h$. This contradicts the equality $q_X(\lambda_{\cE,\cF},h+s_{\min} f)= 0$. 

Next, we prove Item~(ii). Suppose that the restriction $\cF_{|{X_t}}$ is $h_t$ slope-unstable  for generic $t\in\PP^n$.  
As before, let ${\mathsf S}\subset([0,1]\cap\QQ)$ be the subset of $s$ such that there   
exists a subsheaf $\cE\subset\cF$, with rank $0<r(\cE)<r(\cF)$, for which~\eqref{comenovem} holds.  We claim that $\mathsf S$ is not empty, and that it has a minimum (N.B.: it does not have a maximum). 

In fact, since  $\cF_{|{X_t}}$ is $h_t$ slope-unstable  for generic $t\in\PP^n$, there exists a subsheaf  $\cE\subset\cF$, with $0<r(\cE)<r(\cF)$, such that $\cE_t\subset\cF_t$ is $h_t$ slope desemistabilizing for generic $t\in\PP^n$. By Lemma~\ref{contreffe}, we have $q_X(\lambda_{\cE,\cF},f)>0$. 
On the other hand $q_X(\lambda_{\cE,\cF},h)<0$ because $\cF$ is $h$ slope-stable. It follows  that $\mathsf S$ is not empty.

 It remains to show that $\mathsf S$ has a minimum. Suppose that~\eqref{comenovem} holds. Since
  $q_X(\lambda_{\cE,\cF},h)<0$  (because $\cF$ is $h$ slope-stable), we get that $q_X(\lambda_{\cE,\cF},f)>0$. Hence the sheaves $\cE\subset\cF$, with $0<r(\cE)<r(\cF)$, such that~\eqref{comenovem} holds for some $s\in[0,1]\cap\QQ$ are exactly those such that  $\cE_{|_{X_t}}\subset \cF_{|_{X_t}}$  is an $h_t$ slope desemistabilizing sheaf of 
 $\cF_{|_{X_t}}$, for the generic  $t\in\PP^n$. 
 
Let  $\wt{X}:=X\times_{\PP^n}\CC(\PP^n)$ be the abelian variety over $\CC(\PP^n)$ obtained from $X$ by base change. We let $\wt{h}$ be the ample divisor on $\wt{X}$ determined by $h$.
A subsheaf $\cE\subset\cF$ on $X$ determines a subsheaf $\wt{\cE}\subset\wt{\cF}$ on $\wt{X}$. 

 Then $\mu_{\wt{h}}(\wt{\cE})>\mu_{\wt{h}}(\wt{\cF})$, i.e.~$\wt{\cE}$  is $\wt{h}$ desemistabilizing for   $\wt{\cF}$ if and only if  $\cE_{|_{X_t}}\subset \cF_{|_{X_t}}$  is an $h_t$ slope desemistabilizing sheaf of 
 $\cF_{|_{X_t}}$, for the generic  $t\in\PP^n$. The set of  $\wt{h}$ desemistabilizing subsheaves  $\cA\subset\wt{\cF}$ is bounded. Given  such a subsheaf, there exists a unique maximal subsheaf 
 $\cE\subset\cF$ such that $\wt{\cE}=\cA$. The set ${\mathsf S}^0$ of $s\in([0,1]\cap\QQ)$ such that~\eqref{comenovem} holds for such a maximal subsheaf is finite (and non empty), by boundedness. Hence there is a minimum $s_{\min}^0$ element of ${\mathsf S}^0$.
 All other subsheaves $\cE\subset\cF$ (with $0<r(\cE)<r(\cF)$) such that $\wt{\cE}\subset\wt{\cF}$ is a 
 $\wt{h}$ desemistabilizing subsheaf, are contained in a maximal subsheaf $\ov{\cE}$, and the quotient $\ov{\cE}/\cE$ is supported on vertical divisors (i.e.~divisors whose image under $\pi$ is a proper subset of $\PP^n$). It follows that $s_{\min}^0$ is also the minimum element of ${\mathsf S}$. 
 
The sheaf  $\cF$ is strictly $((1-s_{\min})h+s_{\min}f)$ slope-semistable by minimality of $s_{\min}$. Let $\cE\subset\cF$ be a subsheaf, with $0<r(\cE)<r(\cF)$  such that 
$q_X(\lambda_{\cE,\cF},(1-s_{\min})h+s_{\min} f)= 0$. By Proposition~\ref{propsemi}  either $-a(\cF)\le q_X(\lambda_{\cE,\cF})< 0$ or $\lambda_{\cE,\cF}=0$.
The latter does not hold because $q_X(\lambda_{\cE,\cF},h)<0$ ($\cF$ is $h$ slope-stable). Hence $-a(\cF)\le q_X(\lambda_{\cE,\cF})< 0$ and thus
 $q_X(\lambda_{\cE,\cF}, f)< 0$ by our hypothesis on $h$. This contradicts the equality $q_X(\lambda_{\cE,\cF},(1-s_{\min})h+s_{\min} f)= 0$. 
\end{proof}
\section{Stable vector bundles on Lagrangian hyperk\"ahlers}\label{stablag}
\subsection{Main result}
\setcounter{equation}{0}
Before  stating the main result we recall that $\cK_{e}^i$ is the moduli space of polarized HK's $(X,h)$ of Type $K3^{[2]}$ with $q(h)=e$ and $h$  has divisibility $i$ (see~\eqref{divuno} and~\eqref{divdue}), which is $1$ if $e\not\equiv 6\pmod{8}$, and is either $1$ or $2$ if $e\equiv 6\pmod{8}$.

The Noether-Lefschetz divisor  $\cN^i_d(e)\subset \cK_{e}^i$  parametrizes $(X,h)$ such that there exists a saturated rank $2$ sublattice $\la h,f\ra\subset H^{1,1}_{\ZZ}(X)$, where $f$ is isotropic and $q(h,f)=d$, see Definition~\ref{ennelagr}. Assume that  $d> 10(e+1)$, that $e\not\divides 2d$ and that $d$ is even if $i=2$. By Proposition~\ref{unicafibr} $\cN^i_d(e)$  is of pure codimension $1$ (in particular non empty), and there exists an open dense subset 
$\cN^i_d(e)^0\subset \cN^i_d(e)$ such that the following holds for $[(X,h)]\in\cN^i_d(e)^0$: there exists one and only one Lagrangian fibration $\pi\colon X\to\PP^2$ 
(modulo automorphisms of $\PP^2$) such that, letting $f:=\pi^{*} c_1(\cO_{\PP^2}(1))$, the lattice $\la h,f\ra$ is as above.  
 Below is the main result of the present section.
\begin{prp}\label{propriostab}
Let $a_0,d$ be  positive integers and $i\in\{1,2\}$. Suppose that $e\not\divides 2d$, that $d$ is even if $i=2$, and that 
\begin{equation}\label{golfangora}
d>\max\left\{\frac{1}{2}a_0(e+1), 10(e+1)\right\}.
\end{equation}
If $[(X,h)]\in\cN_e^i(d)^0$ is generic the following hold:  
\begin{enumerate}
\item
Let  $\cE$ be an $h$ slope-stable vector bundle on $X$ such that 
\begin{enumerate}
\item[(a)]
$a(\cE)\le a_0$, where $a(\cE)$ is as in Definition~\ref{adieffe}, 
\item[(b)]
there exists an integer $m$ such that $r(\cE)=(mi)^2$, $c_1(\cE)=m h$, and $\gcd\{mi,\frac{d}{i}\}= 1$.
\end{enumerate}
Then
 the restriction of $\cE$ to a generic fiber of  the associated Lagrangian fibration $\pi\colon X\to\PP^2$ is slope-stable.
\item
If $\cF,\cG$ are  $h$ slope-stable vector bundles on $X$ such that Items~(a) and (b)  hold for $\cE=\cF$ and $\cE=\cG$, then for generic $z\in\PP^2$ the restrictions of $\cF$ and $\cG$ to $\pi^{-1}(z)$ are isomorphic. 
\end{enumerate}
\end{prp}
\begin{rmk}
Regarding Item~(b) of Proposition~\ref{propriostab}: according to Proposition~\ref{restrango} we always have
 $r(\cE)\divides (mi)^2$, hence the equality   is an extremal case. 
\end{rmk}
\subsection{Preliminary results}
\setcounter{equation}{0}
\begin{lmm}\label{nocamere}
Let $(\Lambda,q)$ be a non degenerate rank $2$ lattice which represents $0$, and hence $\disc(\Lambda)=-d^2$ where $d$ is a strictly positive integer. Let $\alpha\in \Lambda$ be primitive isotropic, and complete it to a basis $\{\alpha,\beta\}$ such that $q(\beta)\ge 0$. If $\gamma\in\Lambda$ has strictly negative square (i.e.~$q(\gamma)<0$) then
\begin{equation}\label{menoenne}
 q(\gamma)\le -\frac{2d}{1+q(\beta)}.
\end{equation}
\end{lmm}
\begin{proof}
There exist integers $x,y$ such that $\gamma=x\alpha+y\beta$. Since $\disc(\Lambda)=-q(\alpha,\beta)^2$ we have $q(\alpha,\beta)=d$. Thus
\begin{equation*}
q(\gamma)=y(2d x+ q(\beta) y).
\end{equation*}
Since $q(\gamma)<0$ and since $x,y$ are integers, we have   
\begin{equation*}
0<|x|,\quad 0<|y|\le |q(\gamma)|,\quad 0<|2d x+ q(\beta) y|\le |q(\gamma)|.
\end{equation*}
It follows that
\begin{equation*}
 2d |x|-q(\beta) |y| \le | 2d x+ q(\beta) y|\le |q(\gamma)|
\end{equation*}
because  $d$ and $q(\beta)$ are non negative.
Hence  
\begin{equation*}
2d\le  2d |x| \le q(\beta) |y|+ | 2d x+ q(\beta) y| \le q(\beta)|q(\gamma)| +|q(\gamma)|=(1+q(\beta))|q(\gamma)|.
\end{equation*}
Since $q(\gamma)<0$ the above inequality is equivalent to~\eqref{menoenne}.
\end{proof}
\begin{prp}\label{numeretti}
Let $(A,\theta)$ be a principally polarized abelian  surface.  
Let $\cF$ be a $\theta$ slope-semistable vector bundle on $A$ such that $c_1(\cF)$ is a multiple of $\theta$ and
$\Delta(\cF)=0$.
Then we can write
\begin{equation}\label{dueg}
r(\cF)=r_0^2 m,\quad c_1(\cF)=r_0 b_0 m \theta, 
\end{equation}
where $r_0,m,b_0$ are integers, the first two are positive, and $\gcd\{r_0,b_0\}=1$.
If $\cF$ is  strictly $\theta$ slope-semistable, i.e.~not slope-stable, then there exists such a decomposition with $m>1$.
\end{prp}
\begin{proof}
If $\cF$ is slope-stable, then it is simple semi-homogeneous by Proposition~\ref{discperp}, and hence we may write~\eqref{dueg} with $x=1$ by Proposition~\ref{potenza}.

Suppose that $\cF$ is strictly $\theta$ slope-semistable. Hence there exists a destabilizing   exact sequence of torsion free sheaves 
\begin{equation}\label{potter}
0\lra \cG\lra \cF\lra \cH\lra 0
\end{equation}
with $\cG$ slope-stable.   Notice that  $\cG$  is locally free because $\cH$ is torsion free. 

Let us prove that $\cG$ is simple semi-homogeneous. Since~\eqref{potter} is slope destabilizing, $0<r(\cG)<r(\cF)$ and $\int_A\lambda_{\cG,\cF}\smile\theta=0$,
where $\lambda_{\cG,\cF}\in H^2(A;\ZZ)$ is defined in~\eqref{pecora}.
Since $\cF$ is slope-semistable, $\cH$ is slope-semistable. Thus  $\Delta(\cG)\ge 0$ and $\Delta(\cH)\ge 0$ by Bogomolov. Now look at Equation~\eqref{lungaeq}: since  $\int_A\lambda_{\cG,\cF}^2\le 0$ by Hodge index, we get that $\Delta(\cG)=\Delta(\cH)=0$ and $\int_A\lambda_{\cG,\cF}^2= 0$. In particular $\cG$ is simple semi-homogeneous by Proposition~\eqref{discperp}
 and $\cH$ is locally free (if it is not locally free then $\cH^{\vee}$ is a slope semistable vector bundle with  $\Delta(\cH^{\vee})<\Delta(\cH)=0$, but this contradicts Bogomolov's inequality). The equality $\int_A\lambda_{\cG,\cF}^2= 0$ gives
 (by the Hodge Index Theorem) that $\lambda_{\cG,\cF}=0$. Thus $c_1(\cG)$ is a multiple of $\theta$, and so is $c_1(\cH)$.

Since the vector bundle $\cH$ is  slope semistable, $c_1(\cH)$  is a multiple of $\theta$ and $\Delta(\cH)=0$, we can iterate this argument to get the following result. Let
\begin{equation*}
0=\cG_0\subsetneq\cG_1\subsetneq\ldots\subsetneq\cG_m=\cF
\end{equation*}
be a Jordan-H\"older filtration  (for slope semistability) of $\cF$. Then for $i\in\{1,\ldots,m\}$ 
the  quotient $\cG_i/\cG_{i-1}$  is a simple semi-homogeneous vector bundle and 
$c_1(\cG_i/\cG_{i-1})$ is a multiple of $\theta$.
Let $i\in\{1,\ldots,m\}$; by Proposition~\ref{potenza} we may write 
\begin{equation*}
r(\cG_i/\cG_{i-1})=r_i^2,\qquad c_1(\cG_i/\cG_{i-1})=r_i b_i \theta, 
\end{equation*}
where  $r_i, b_i$ are integers, $r_i>0$ and $\gcd\{r_i,b_i\}=1$. Let $i,j\in\{1,\ldots,m\}$; equating the slopes of $\cG_i/\cG_{i-1}$ and 
$\cG_j/\cG_{j-1}$ we get that
\begin{equation}\label{riduco}
\frac{ b_i}{r_i}=\frac{b_j}{r_j}.
\end{equation}
Since  $\gcd\{r_i,b_i\}=\gcd\{r_j,b_j\}=1$ it follows that $r_i=r_j$ and $b_i=b_j$. Thus $r(\cF)=m r_0^2$ and $c_1(\cF)=m r_0 b_0\theta$ where $r_0=r_i$ and $b_0=b_i$ for all $i\in\{1,\ldots,m\}$.
\end{proof}
\begin{crl}\label{caravilla}
Let $(A,\theta)$ be a principally polarized abelian  surface.  
Let $\cF$ be a $\theta$ slope-semistable vector bundle on $A$ such that 
$\Delta(\cF)=0$.
If $r(\cF)=r_0^2$, $c_1(\cF)=r_0 b_0 \theta$ where $r_0, b_0$ are \emph{coprime} integers, then
 $\cF$ is  $\theta$ slope-stable.
\end{crl}
\begin{proof}
By contradiction. Suppose that $\cF$ is not   $\theta$ slope-stable. By Proposition~\ref{numeretti} we may write
$r(\cF)=s_0^2 m$, $c_1(\cF)=s_0 c_0 m\theta$ where $s_0,m,c_0$ are integers (with $s_0,m>0$), $s_0,c_0$ are coprime and $m>1$. It follows that $s_0 b_0=c_0 r_0$. Since $\gcd\{r_0,b_0\}=1$ and $\gcd\{s_0,c_0\}=1$, we get that $r_0=s_0$ and hence $m=1$. This is a contradiction.
\end{proof}
\subsection{Proof of Item~(1) of Proposition~\ref{propriostab}}
\setcounter{equation}{0}
First we prove that $h$ is $a_0$-suitable (see Definition~\ref{suipol}). Suppose first that $\rho(X)=2$, i.e.
$H^{1,1}_{\ZZ}(X)=\la h,f\ra$, where $f:=\pi^{*}c_1(\cO_{\PP^2}(1))$. Apply Lemma~\ref{nocamere} to $\Lambda:=H^{1,1}_{\ZZ}(X)$, $\alpha=f$ and $\beta=h$: by~\ref{golfangora} we get that there are no $\xi\in H^{1,1}_{\ZZ}(X)$ such that $-a_0\le q(\xi)<0$. Hence every ample divisor on $X$ is $a_0$-suitable. 

Once we know that $h$ on $X$ is $a_0$-suitable if $\rho(X)=2$, it follows that the set of $[(X,h)]\in\cN^i_e(d)^0$ such that
 $h$  is \emph{not} $a_0$-suitable belongs to  the intersection of $\cN^i_e(d)^0$ with a finite union of Noether-Lefschetz 
 divisors in  $\cK^i_e$. In fact suppose that $h$  is \emph{not} $a_0$-suitable on $X$. Then there exists 
 $\gamma\in H^{1,1}_{\ZZ}(X)$ such that
 \begin{equation}\label{trediseg}
-a_0\le q(\gamma)<0,\quad q(\gamma,h)>0,\quad q(\gamma,f)<0.
\end{equation}
Let $B$ be the (finite) index of $\la h,f\ra\oplus (\la h,f\ra^{\bot}\cap H^{1,1}_{\ZZ}(X))$ in $H^{1,1}_{\ZZ}(X)$.
Then
\begin{equation}
\gamma=\frac{\gamma_1}{B}+\frac{\gamma_2}{B},\qquad \gamma_1\in\la h,f\ra,\quad \gamma_2\in\la h,f\ra^{\bot}.
\end{equation}
By the last two inequalities in~\eqref{trediseg} we have $q(\gamma_1)<0$. Hence by the first inequality in~\eqref{trediseg} it follows that there exists a positive $M$ independent of $(X,h)$ such that $-M\le q(\gamma_2)<0$. Hence the moduli point of $(X,h)$ belongs to  the intersection of $\cN^i_e(d)^0$ with a finite union of Noether-Lefschetz 
 divisors in  $\cK^i_e$, as claimed.

We have proved that if $(X,h)$  represents a generic point of  $\cN_e^i(d)^0$, then $h$ is $a_0$-suitable, and hence $a(\cE)$-suitable because $a(\cE)\le a_0$. Let $A$ be a  generic (smooth) fiber of $\pi$. By 
Proposition~\ref{lagstab} the restriction of $\cE$ to $A$ is slope-semistable with respect to the restriction of $h$. 

 We claim that the hypotheses of Corollary~\ref{caravilla} are satisfied by  $\cF:=\cE_{|A}$. In fact $\Delta(\cF)=0$   because $\cE$ is modular, see Lemma~\ref{zerene}. 
Moreover  the restriction of $h$ to $A$ is a multiple of a principal polarization $\theta$ by Theorem~1.1 in~\cite{wieneck1}. From the formula $\int_A h^2=\int_X h^2\smile f^2=2 q(h,f)^2=2d^2$ it follows that $h|_A=d\theta$. Hence $r(\cF)=(mi)^2$ and
$c_1(\cF)=m\cdot d \theta=(mi)\frac{d}{i}$. It follows that the hypotheses of Corollary~\ref{caravilla} are satisfied and hence   $\cF$ is slope-stable. 
\subsection{Proof of Item~(2) of Proposition~\ref{propriostab}}
\setcounter{equation}{0}
For $z\in\PP^2$ we let $A_z:=\pi^{-1}(z)$ and $\cF_z:=\cF_{|A_z}$, $\cG_z:=\cG_{|A_z}$.
By Item~(2) of Proposition~\ref{propriostab} there exists an open dense $U\subset\PP^2$ such that for $z\in U$ the vector bundles  $\cF_z$ and $\cG_z$ are 
both slope-stable. We claim that if  $z\in U$ then $\cF_z$ and $\cG_z$ are simple semi-homogeneous vector bundles. In fact they are simple because they are slope-stable, and they are semi-homogeneous by Lemma~\ref{zerene} and Proposition~\ref{discperp}. Let  $z\in U$. By Theorem~7.11 in~\cite{muksemi} the set
\begin{equation*}
V_z:=\{[\xi]\in A_z^{\vee} \mid \cF_z\cong \cG_z\otimes\xi\}
\end{equation*}
is not empty, and hence it has cardinality $r(\cG)^2$ by Proposition~7.1 op.~cit. Clearly $V_z$ is invariant under the monodromy action of $\pi_1(U,z)$. 
Now notice that $V_z\subset A_z[(mi)^2]$ because $\cF_z$ and $\cG_z$ have rank $(mi)^2$ and isomorphic determinants. 
Hence  by Corollary~\ref{modinv} we have $V_z=A[mi]$. Thus $0\in V_z$, and therefore $\cF_z\cong \cG_z$.

\section{Basic modular sheaves on Hilbert squares of $K3$'s}\label{esempi}
\setcounter{equation}{0}
\subsection{Main results}
\setcounter{equation}{0}
Let  $S$ be a smooth projective surface. Let $X_n(S)\to S^n$ be the blow up of the big diagonal, i.e.~the $n$-th isospectral Hilbert scheme of $S$, see Definition 3.2.4 and Proposition 3.4.2 in~\cite{haiman}. The complement of the big diagonal in $S^n$ is identified with a dense open subset $U_n(S)$ of $X_n(S)$, and the natural map $U_n(S)\to S^{[n]}$ extends to a regular map $p\colon X_n(S)\to S^{[n]}$ (this follows from Proposition 3.4.2 in~\cite{haiman}). 
  Let $\tau\colon X_n(S)\to S^n$ be the blow up map. We let $q_i\colon X_n(S)\to S$ be the composition of $\tau$ and the $i$-th projection $S^n\to S$. Given a locally free sheaf $\cF$ on $S$, let
\begin{equation*}
X_n(\cF):=q_1^{*}(\cF)\otimes\ldots\otimes q_n^{*}(\cF).
\end{equation*}
The action of the symmetric group $\cS_n$ on $S^n$ by permutation of the factors maps the big diagonal to itself, and hence lifts to  an action  $\rho_n\colon \cS_n\to\Aut(X_n(S))$.  The latter action lifts to a natural action $\rho_n^{+}$ on  $X_n(\cF)$. There is also a twisted  action $\rho_n^{-}=\rho_n^{+}\cdot\chi$ where  $\chi\colon\cS_n\to\{\pm 1\}$ is the sign character. Since $\rho_n$  maps to itself any fiber of  $p\colon X_n(S)\to S^{[n]}$,  $\rho_n^{\pm}$  descends to an action  $\ov{\rho}_n^{\pm}\colon \cS_n \to \Aut(p_{*}X_n(\cF))$. 
\begin{dfn}
Let $\cF^{\pm}[n]\subset p_{*}X_n(\cF)$ be the sheaf of $\cS_n$-invariants for $\ov{\rho}_n^{\pm}$.
\end{dfn}
Below is the first main result of the present section.
\begin{prp}\label{yaufever}
Let $S$ be a projective  $K3$ surface, and let $\cF$ be a locally free sheaf on $S$ such that $\chi(S,\End(\cF))=2$. 
Then $\cF[2]^{\pm}$ is a locally free modular sheaf of rank $r(\cF)^2$, with
\begin{eqnarray}
\Delta(\cF[2]^{\pm}) & = & \frac{r(\cF[2]^{\pm})(r(\cF[2]^{\pm})-1)}{12}c_2(S^{[2]}),  \label{disceffe} \\
d(\cF[2]^{\pm}) & = & 5\cdot {r(\cF[2]^{\pm})\choose 2}, \label{didieffe} \\
a(\cF[2]^{\pm}) & = & \frac{5}{8}r(\cF)^6(r(\cF)^2-1). \label{adieffebis}
\end{eqnarray}
(Recall that $d(\cE)$ is defined by the equality in~\eqref{fernand}.)  
\end{prp}
The proof of Proposition~\ref{yaufever} is given in Subsection~\ref{subsec:lucilla}.
\begin{rmk}
If $S$ is a $K3$ surface, and  $\cF$ is a locally free sheaf on $S$ such that $\chi(S,\End (\cF))\not=2$, then $\cF[2]^{\pm}$ is \emph{not} modular.  
\end{rmk}
The second main result of the section is the following.
\begin{prp}\label{prp:seirigido}
Let $S$ be a projective  $K3$ surface. Let $\cF$ be a locally free sheaf on $S$ which is spherical, 
i.e.~such that $h^p(S,End^0(\cF))=0$ for all $p$,  where $End^0(\cF)\subset End(\cF)$ is the  subsheaf of traceless endomorphisms. Then for all $p$ we have
\begin{equation}
h^p(S^{[2]},End^0(\cF[2]^{\pm})) =0.
\end{equation}

\end{prp}
The proof of Proposition~\ref{prp:seirigido} is in Subsection~\ref{subsec:priscilla}. 

Below is  a remarkable consequence of Proposition~\ref{prp:seirigido}.
\begin{crl}\label{iacman}
Let $S$ be a projective  $K3$ surface and let $\cF$ be a locally free sheaf on $S$ which is spherical. Then the natural map between deformation spaces  $\Def(S^{[2]},\cF[2]^{\pm})\lra \Def(S^{[2]},\det\cF[2]^{\pm})$ is smooth.
\end{crl}
\begin{proof}
This follows from Proposition~\ref{prp:seirigido} and the main result of~\cite{man-iac-pairs}.
\end{proof}

\subsection{Another description of $\cF[2]^{\pm}$}\label{subsec:altdes}
\setcounter{equation}{0}
A different definition of the sheaf $\cF[2]^{-}$ was given in~\cite{dhov-journal}. Here we recall that construction and we give the analogous construction of $\cF[2]^{+}$. 

The isospectral Hilbert scheme $X_2(S)$ is the 
 the blow up of the diagonal in $S\times S$; we will denote it by $\wt{S\times S}$. Let $E$ be the exceptional divisor of the blow up map $\tau\colon\wt{S\times S}\to S^2$, and let  $e\in H^2(\wt{S\times S})$ its Poincar\'e dual. We let $\tau_E\colon E\to S$ be the restriction of $\tau$ to $E$.  
Let $\iota\in \Aut(\wt{S\times S})$ be the involution lifting the involution of $S^2$ exchanging the factors. Then  $S^{[2]}$ is the quotient of $\wt{S\times S}$ by the group $\la \iota\ra$, and $p\colon \wt{S\times S}\to S^{[2]}$ is the quotient map.   We recall that  the map  
$q_i\colon  \wt{S\times S}\to S$   for $i\in\{1,2\}$ is the composition of $\tau$ and the $i$-th projection    $S\times S\to S$.

The natural map $f^{\pm}\colon p^*(\cF[2]^{\pm})\to q_1^*\cF\otimes q_2^*\cF$ is an isomorphism away from $E$, in particular it is injective because $p^*(\cF[2]^{\pm})$ is torsion free. In order to write out the cokernel we notice that
there are  surjective morphisms
\begin{equation}\label{evpiu}
 q_1^*\cF\otimes q_2^*\cF \overset{\ev^+}{\lra} \tau_E^*\Sym^2\!\cF,\qquad 
 q_1^*\cF\otimes q_2^*\cF \overset{\ev^-}{\lra} \tau_E^*\bw 2\!\cF
\end{equation}
obtained by evaluating along $E$ and then projecting onto the symmetric/antisymmetric part of
$(q_1^*\cF\otimes q_2^*\cF)\vert_{ E}=\tau_E^*( \cF\otimes\cF)$. Let $\iota\colon E\hra\wt{S\times S}$ be the inclusion. 
\begin{prp}[See Lemma~4.2 in~\cite{dhov-journal}]\label{ledescend}
 The sheaves $\cF[2]^{\pm}$ are locally free of rank $r(\cF)^2$  and the following are exact sequences:
 \begin{equation}\label{eccopi}
0\lra p^*(\cF[2]^{+})\overset{f^{+}}{\lra} q_1^*\cF\otimes q_2^*\cF  \overset{\ev^{-}}{\lra}  \iota_{*}\tau_E^*(\bw 2\!\cF)\lra 0,
\end{equation}
 \begin{equation}\label{eqcK}
0\lra p^*(\cF[2]^{-})\overset{f^{-}}{\lra} q_1^*\cF\otimes q_2^*\cF \overset{\ev^+}{\lra}  \iota_{*}\tau_E^*\Sym^2\!\cF\lra 0.
\end{equation}
\end{prp}
\begin{proof}
Away from $E$ the sheaves $\iota_{*}\tau_E^*(\bw 2\!\cF)$ and $ \iota_{*}\tau_E^*\Sym^2\!\cF$ are zero and the maps $f^{\pm}$ are isomorphisms. Hence~\eqref{eccopi} and~\eqref{eqcK} are exact away from $E$. In particular $\cF[2]^{\pm}$ is locally free of rank $r(\cF)^2$ away from $E$.

Let $x\in E$; then the subscheme $y:=p(x)\subset S$ is non reduced  and hence it is supported at a single point $y_{\rm red}$.
Let  
$h\in\cO_{\wt{S\times S},x}$ be a local generator of the ideal of $E$. Let $\cF(y_{\rm red})$ be the fiber of $\cF$ at $y_{\rm red}$. The $\pm 1$ eigenspaces for the action of $\rho_2^{+}$ on 
$p_{*}(q_1^*\cF\otimes q_2^*\cF)_y$ are respectively
 \begin{equation*}
\left(\Sym^2\cF(y_{\rm red})\otimes\cO_{S^{[2]},y}\right)\oplus \left(\bigwedge^2\cF(y_{\rm red})\otimes\cO_{S^{[2]},y}\cdot h \right)
\end{equation*}
and
 \begin{equation*}
\left(\bigwedge^2\cF(y_{\rm red})\otimes\cO_{S^{[2]},y}\right)\oplus \left(\Sym^2\cF(y_{\rm red})\otimes\cO_{S^{[2]},y}\cdot h \right).
\end{equation*}
Thus  $\cF[2]^{+}_x$ is free of rank $r(\cF)^2$, and we get that~\eqref{eccopi} is exact at $x$. Since the $\pm$ eigenvalues of the  action of $\rho_2^{-}$ are the $\mp$ eigenvalues of the  action of $\rho_2^{+}$ we get also that $\cF[2]^{-}_x$ is free of rank $r(\cF)^2$ and that~\eqref{eqcK} is exact at $x$. 
\end{proof}
\subsection{Preliminaries on $K3^{[n]}$}
\setcounter{equation}{0}
Let $\mu_n\colon H^2(S)\to H^2(S^{[n]})$ 
be the composition of the  natural symmetrization map $H^2(S)\to H^2(S^{(n)})$ and the pull-back $H^2(S^{(n)})\to H^2(S^{[n]})$ defined by  the Hilbert-Chow map $S^{[n]}\to S^{(n)}$. Let $\Delta^{[n]}\subset S^{[n]}$ be the prime divisor parametrizing non reduced subschemes. The class 
$\cl(\Delta^{[n]})$ is divisible by $2$ in the integral cohomology of $S^{[n]}$;  let $\delta_n\in H^{1,1}_{\ZZ}(S^{[n]})$ be the unique class such that $2\delta_n=\cl(\Delta^{[n]})$. We have an orthogonal decomposition for the BBF quadratic form
\begin{equation*}
H^2(S^{[n]};\ZZ)=\mu_n(H^2(S;\ZZ))\oplus\ZZ\delta_n.
\end{equation*}
Let  $q$ be the BBF form of $S^{[n]}$. Then for $\alpha\in H^2(S)$ we have
\begin{equation}\label{bibifu}
q(\mu_n(\alpha))=\int\limits_{S}\alpha^2,\quad q(\mu_n(\alpha),\delta_n)=0,\quad q(\delta_n)=-2(n-1).
\end{equation}
We will deal with $S^{[2]}$. In order to simplify notation we will drop the subscripts of $\delta_2$ and $\mu_2$. 
 We go through a few formulas that will be needed in the proof that $\cF[2]^{\pm}$ is a modular sheaf. Let $\eta\in H^4(S;\ZZ)$ be the orientation class. We claim that
\begin{eqnarray}
p^{*}\left(x\mu(c_1(\cF))-y\delta\right) & = & x(q_1^{*}c_1(\cF)+q_2^{*}c_1(\cF))-ye, \label{turandot} \\ 
 p^{*}c_2(S^{[2]}) & = & 24(q_1^{*}\eta+q_2^{*}\eta)-3e^2.  \label{tirocidue}
\end{eqnarray}
In fact~\eqref{turandot} follows directly from the definitions, and~\eqref{tirocidue} is the last equation on p.~84 of~\cite{devo}. 
Equation~\eqref{tirocidue} gives 
\begin{equation}\label{prodcidue}
\int_{S^{[2]}}c_2(S^{[2]})^2=828, \qquad \int_{S^{[2]}}c_2(S^{[2]})\smile\alpha^2=30 q(\alpha),\quad \alpha \in H^2(S^{[2]}).
\end{equation}
\begin{lmm}
Let $S$ be a $K3$ surface. Let $\alpha\in H^2(S)$, and  let $\alpha^2=2m_0\eta$.  Then
\begin{equation}\label{isouno}
%
%
2\left(q_1^{*}\eta\smile q_2^{*}\alpha+q_1^{*}\alpha\smile q_2^{*}\eta\right)+ 
\left(q_1^{*}\alpha+q_2^{*}\alpha\right)\smile e^2  =  0. 
\end{equation}
\end{lmm}
\begin{proof}
Since the cohomology of $\wt{S\times S}$ has no torsion it suffices to check that the cup product of the left hand side of~\eqref{isouno}  with any class in $H^2(\wt{S\times S})$ vanishes. Thus we must take the cup product with $q_{i}^{*}\beta$ where $i\in\{1,2\}$ and $\beta\in H^2(S)$, and with $e$. The easy computations are left to the reader.
\end{proof}
\subsection{Chern classes of $\cF[2]^{\pm}$}\label{subsec:lucilla}
\setcounter{equation}{0}
\begin{prp}\label{estende}
Let $S$ be a $K3$ surface, and let $\cF$ be a locally free sheaf of rank $r_0$ on $S$ such that $\chi(S,End (\cF))=2$. 
Let $h^{\pm}\in H^{1,1}_{\QQ}(S^{[2]})$ be defined by
\begin{equation}\label{accapiumeno}
h^{\pm}:=\mu(c_1(\cF))-\frac{r_0\mp 1}{2}\delta.
\end{equation}
Then 
\begin{eqnarray}
\ch_0(\cF[2]^{\pm}) & = & r_0^2, \label{parzero} \\
\ch_1(\cF[2]^{\pm}) & = & r_0 h^{\pm}, \label{paruno} \\
\ch_2(\cF[2]^{\pm})  & = & \frac{1}{2}(h^{\pm})^2   -\frac{r_0^2-1}{24}c_2(S^{[2]}). \label{pardue}
\end{eqnarray}
\end{prp}
\begin{proof}
Of course~\eqref{parzero} needs no proof.   
Let $\ch_1(\cF)^2=2m_0\eta$ where 
$\eta\in H^4(S;\ZZ)$ is the orientation class. Since $\chi(S,End (\cF))=2$, 
 HRR gives that
\begin{equation}\label{sottsass}
2r_0\ch_2(\cF)=\ch_1(\cF)^2-2(r_0^2-1)\eta=(2m_0-2(r_0^2-1))\eta.
\end{equation}
A straightforward computation shows that
\begin{equation}\label{reldiqu}
2m_0=q(h^{\pm})+\frac{(r_0\mp 1)^2}{2}.
\end{equation}
Since  the pull-back $p^{*}\colon H(S^{[2]};\ZZ)\to H^2(\wt{S\times S};\ZZ)$ is injective, we work on $\wt{S\times S}$. 
By~\eqref{eccopi} and GRR, we have 
\begin{multline}\label{carattere}
\scriptstyle  p^{*}\ch(\cF[2]^{+})    =  \\
\scriptstyle = q_1^{*}\ch(\cF)\cdot q_2^{*}\ch(\cF)- 
\iota_{*}\left(\tau_E^{*}({r_0\choose 2}+ (r_0-1)\ch_1(\cF)+(r_0-2)\ch_2(\cF)+\frac{1}{2}\ch_1(\cF)^2)\right)\cdot \left(1-\frac{1}{2}e+\frac{1}{6}e^2-\frac{1}{24}e^3\right) = \\
 \scriptstyle  = q_1^{*}\left(r_0+\ch_1(\cF)+\ch_2(\cF)\right)\cdot q_2^{*}\left(r_0+\ch_1(\cF)+\ch_2(\cF)\right) - \\
 \scriptstyle -\left({r_0\choose 2}+\frac{r_0-1}{2}\sum\limits_{i=1}^2 q_i^{*}\ch_1(\cF)+
 \frac{1}{4}\sum\limits_{i=1}^2 (q_i^{*}((2r_0-4)\ch_2(\cF)+\ch_1(\cF)^2)\right)
\cdot\left(e-\frac{1}{2}e^2+\frac{1}{6}e^3-\frac{1}{24}e^4\right).
\end{multline}
(Cup product is denoted by $\cdot$ in order to save space.) Equation~\eqref{paruno} for $\cF[2]^{+}$ follows at once. Using~\eqref{sottsass},   we get that
\begin{eqnarray*}
p^{*}\ch_2(\cF[2]^{+})  =   -(r_0^2-1)(q_1^{*}\eta+q_2^{*}\eta)+\frac{1}{2}(q_1^{*}\ch_1(\cF)^2+q_2^{*}\ch_1(\cF)^2)+ \\
 +q_1^{*}\ch_1(\cF)\cdot q_2^{*}\ch_1(\cF)-\frac{r_0-1}{2}e\cdot(q_1^{*}\ch_1(\cF)+q_2^{*}\ch_1(\cF))+\frac{1}{2}{r_0\choose 2}e^2.
\end{eqnarray*}
By~\eqref{turandot} and \eqref{tirocidue}, we get that~\eqref{pardue} holds for $\cF[2]^{+}$. 

The computations  for $\cF[2]^{-}$   are similar. 
\end{proof}
Now we prove Proposition~\ref{yaufever}. The sheaf $\cF[2]^{\pm}$ is locally free of rank $r(\cF)^2$ by Proposition~\ref{ledescend}. Equation~\eqref{disceffe} holds by Proposition~\ref{estende}. It follows that $\cF[2]^{\pm}$ is modular (see Remark~\ref{semphodge}). Equations~\eqref{didieffe} and~\eqref{adieffebis} follow from~\eqref{disceffe} and  the second equality in~\eqref{prodcidue}.

\subsection{Cohomology groups via the MacKay correspondence}\label{subsec:priscilla}
\setcounter{equation}{0}
We keep notation introduced in Subsection~\ref{subsec:altdes}. In particular
$\tau\colon\wt{S\times S}\to S^2$ is  the blow up of the diagonal and $p\colon \wt{S\times S}\to S^{[2]}$ is the quotient map for the   action of $\ZZ/(2)$ on $\wt{S\times S}$ which lifts the permutation action on $S^2$.  Let $D^{b}_{\ZZ/(2)}(S^2)$ be the bounded derived category of the (abelian) category of $\ZZ/(2)$-equivariant coherent sheaves on $S^2$. By the MacKay correspondence proved by Haiman and Bridgeland-King-Reid 
the functor
\begin{equation}
p_{*}^{\ZZ/(2)}\circ \tau^{*}\colon D^{b}_{\ZZ/(2)}(S^2)\lra D^{b}(S^{[2]})
\end{equation}
is an equivalence, see Proposition~2.8 in~\cite{krug}. (Here $p_{*}^{\ZZ/(2)}$ is the derived functor  of  the  functor $\Coh(\wt{S\times S})\to \Coh(S^{[2]})$ mapping $\cF$ to the $\ZZ/(2)$-invariant subsheaf of $p_{*}(\cF)$).
We have 
\begin{equation}
p_{*}^{\ZZ/(2)}\circ \tau^{*}(\cF^{\boxtimes 2})=\cF^{+},\qquad 
p_{*}^{\ZZ/(2)}\circ \tau^{*}(\cF^{\boxtimes 2}\otimes{\mathfrak a})=\cF^{-},
\end{equation}
where ${\mathfrak a}$ is the sign representation of $\ZZ/(2)$. 
It follows by the K\"unneth formula that
\begin{equation}\label{unduetre}
\Ext^{*}(\cF^{\pm},\cF^{\pm})\cong \Ext^{*}(\cF^{\boxtimes 2},\cF^{\boxtimes 2})^{\ZZ/(2)}\cong
\Sym^2\left(\Ext^{*}(\cF,\cF)\right)
\end{equation}
Now we prove Proposition~\ref{prp:seirigido}. Since  $H^{p}(S,End(\cF))=H^{p}(S,End^0(\cF))\oplus H^p(S,\cO_S)$ and $S$ is a $K3$, the cohomology  space $\Ext^{p}(\cF,\cF)$ is $1$ dimensional if $p\in\{0,2\}$ and vanishes otherwise. Hence
\begin{equation}
\Sym^2\left(\Ext^{*}(\cF,\cF)\right)=\Sym^2\left(\CC[0]\oplus\CC[-2]\right)=\CC[0]\oplus\CC[-2]\oplus\CC[-4].
\end{equation}
By~\eqref{unduetre} we get that $H^{p}(S^{[2]},End(\cF))$  is $1$ dimensional if $p\in\{0,2,4\}$ and vanishes otherwise.  Proposition~\ref{prp:seirigido} follows because  $h^{p}(S^{[2]},\cO_{S^{[2]}})=1$ for   $p\in\{0,2,4\}$ and 
$H^{p}(S^{[2]},End(\cF^{\pm}))=H^{p}(S^{[2]},End^0(\cF^{\pm}))\oplus H^p(S^{[2]},\cO_{^{[2]}})$.

\section{Basic modular sheaves on the Hilbert square of an elliptic $K3$}\label{ambrose}
\subsection{Contents of the section}
\setcounter{equation}{0}
We will study the vector bundle $\cF[2]^{+}$ for $\cF$ a spherical vector bundle on a an elliptic $K3$ surface $S$ with Picard number $2$ (analogous results hold for $\cF[2]^{-}$). Given positive integers $e,r_0,i$ satisfying the hypotheses of Theorem~\ref{unicita}, one can choose suitable $S$ and $\cF$ such that the equations in~\eqref{ele} 
hold  for $\cE:=\cF[2]^{+}$ - see Subsection~\ref{traduzione}. Next notice that there is a Lagrangian fibration 
$\pi\colon S^{[2]}\to\PP^2$ associated to the elliptic  fibration $S\to\PP^1$. 
In Section~\ref{subsec:restringo} we analyze the restriction of $\cF[2]^{\pm}$ to  (scheme theoretic) fibers of $\pi$. The key result   is Proposition~\ref{acca20}, which states that the restriction to every Lagrangian fiber is simple. Along the way we make another key observation:  the restriction to a generic Lagrangian fiber is slope stable, see Proposition~\ref{semgen}.
\subsection{Elliptic $K3$ surfaces and stable rigid vector bundles}\label{vectlagr}
\setcounter{equation}{0}
We recall the notions of Mukai vector and Mukai pairing for a $K3$ surface $S$.  If $\cF$ is a sheaf over $S$, the \emph{Mukai vector} of $\cF$ is $v(\cF):=\ch(\cF)\td(S)^{1/2}$. Moreover the  bilinear symmetric \emph{Mukai pairing} $\la \cdot,\cdot\ra$ on $H(S)$ has the following property: if  $\cF,\cE$ 
are sheaves on $S$ then
\begin{equation}\label{mukpair}
\la v(\cF),v(\cE)\ra=-\chi(\cF,\cE):=-\sum_{i=0}^2(-1)^i\dim\Ext^i(\cF,\cE).
\end{equation}
Let $S$ be a $K3$ surface with an elliptic fibration $S\to \PP^1$; we let $C$ be a  fiber of the elliptic fibration.  The 
claim below follows from surjectivity of the period map for $K3$ surfaces.
\begin{clm}\label{eccoell}
Let $m_0,d_0$ be positive natural numbers. There exist $K3$ surfaces $S$  with an elliptic fibration $S\to \PP^1$ such that 
\begin{equation}\label{neronsevero}
H^{1,1}_{\ZZ}(S)=\ZZ[D]\oplus\ZZ[C], \quad
D\cdot D=2m_0,\quad D\cdot C=d_0.
\end{equation}
\end{clm}
The result below provides us with stable vector bundles $\cF$ on elliptic $K3$ surfaces such that $\cF[2]^{\pm}$ has good properties - see Proposition~\ref{acca20}.
\begin{prp}\label{rigsuk}
Let $m_0,r_0,s_0$ be positive integers such that  
 $m_0+1=r_0s_0$. Suppose that $d_0$ is an integer coprime to $r_0$, and that
 \begin{equation}\label{digrande}
d_0>\frac{(2m_0+1)r_0^2(r_0^2-1)}{4}.
\end{equation}
Let $S$ be an elliptic $K3$ surface as in Claim~\ref{eccoell}.
 Then there exists a vector bundle $\cF$ on $S$ such that the following hold:
 \begin{enumerate}
\item
 $v(\cF)=(r_0,D,s_0)$,
\item
$\chi(S,End(\cF))=2$,
\item
$\cF$  is $L$ slope-stable for any polarization $L$ of $S$, 
\item
and the restriction of $\cF$ to \emph{every} fiber of  the  elliptic fibration $S\to \PP^1$ is slope-stable. 
\end{enumerate}
(Notice that every fiber is irreducible by our assumptions on 
$\NS(S)$, hence slope-stability of a sheaf on a fiber is well defined, i.e.~independent of the choice of a polarization.)
\end{prp}
\begin{proof}
(1): The Mukai vector $v=(r_0,D,s_0)\in H(S)$ has square $-2$. Let $L_0$ be  a polarization of $S$. By Theorem 2.1 in~\cite{gulexe} there exists an  $L_0$ slope-semistable  vector bundle $\cF$ on $S$ with $v(\cF)=(r_0,D,s_0)$.

(2): $\chi(S,End(\cF))=2$ because $v(\cF)^2=-2$.

(3): We claim that $\cF$ is actually   $L_0$ slope-stable, and that it is $L$ slope-stable for \emph{any} polarization $L$ of $S$. This follows from the well known results on the stability chamber decomposition of $\Amp(S)$ which have been extended to arbitrary HK varieties  in Section~\ref{camere}. In fact  $v(\cF)^2=-2$ and~\eqref{dueversioni}   give that $\Delta(\cF)=2(r_0^2-1)$. It follows that $a(\cF)=\frac{r_0^2(r_0^2-1)}{2}$, and hence
by Lemma~\ref{nocamere} and~\eqref{digrande}   there is no $a(\cF)$ wall. Thus
there is a single $a(\cF)$-chamber.  

(4): By Proposition~\ref{lagstab}  the restriction of  $\cF$ to a generic  fiber $C$ of the elliptic fibration is slope-semistable (because there is a single $a(\cF)$-chamber). Since  $d_0$, which is the degree of $\cF_{|C}$,  is  coprime to $r_0$, which is the rank of $\cF_{|C}$, it follows that the 
 restriction of $\cF$ to a generic  fiber $C$ is slope-stable. Suppose that there exists a fiber $C_0$ such that $\cF_{|C_0}$ is not slope-stable. Then $\cF_{|C_0}$ is slope-unstable because $d_0$ is coprime to $r_0$. Let 
$\cF_{|C_0}\twoheadrightarrow\cB$ be a desemistabilizing quotient, i.e.~$0<r(\cB)<r$, and 
$\mu(\cB)-\mu(\cF_{|C_0})<0$. Let $\cE$ be the elementary modification of $\cF$ associated to the quotient $\cB$, i.e.~the (torsion free) sheaf on $S$ fitting into the exact sequence 
\begin{equation}
0\lra \cE\lra \cF\lra i_{0,*}\cB \lra 0,
\end{equation}
where $i_0\colon C_0\hra S$ is the inclusion map. Then 
\begin{equation}
v(\cE)^2=v(\cF)^2+2r_0\cdot r(\cB)\cdot\left(\mu(\cB)-\mu(\cF_{|C_0})\right)<v(\cF)^2=-2.
\end{equation}
By~\eqref{mukpair} it follows that $\chi(\cE,\cE)>2$. On the other hand, since $\cE$ is isomorphic to $\cF$ outside $C_0$,  the restriction of $\cE$ to a generic elliptic fiber is slope-stable, and this implies that $\Hom(\cE,\cE)=\CC\Id_{\cE}$. By Serre duality it follows that  $\dim\Ext^2(\cE,\cE)=1$. The last two facts  contradict the inequality $\chi(\cE,\cE)>2$.
\end{proof}
\subsection{Dictionary}\label{traduzione}
\setcounter{equation}{0}
We show that, given positive integers $i,e,r_0$ satisfying the hypotheses of Theorem~\ref{unicita}, we get a vector bundle  $\cE$ on the Hilbert square of a suitable $K3$ surface such that the equations in~\eqref{ele} hold. First we need the result below. The elementary proof is  left to the reader.
\begin{lmm}\label{champollion}
Let $i\in\{1,2\}$. Let  $e,r_0$ be positive natural numbers such  that $r_0\equiv i\pmod{2}$ and~\eqref{econ} holds. Let 
\begin{equation}\label{emmezero}
m_0:=
\begin{cases}
\frac{e}{2}+\frac{(r_0-1)^2}{4} & \text{if $r_0$ is odd,} \\
\frac{e}{8}+\frac{(r_0-1)^2}{4} & \text{if $r_0$ is even.} 
\end{cases}
\end{equation}
($m_0$ is an integer by~\eqref{econ}.) There exists an integer $s_0$ such that $m_0+1=r_0 s_0$. 
\end{lmm}
Let $i,e,r_0,m_0$ be as in Lemma~\ref{champollion}.
Suppose that $d_0$ is an integer  coprime to $r_0$  such that~\eqref{digrande} holds. Let $S$ be an elliptic $K3$ surface as in Claim~\ref{eccoell}, and let $\cF$ be a vector bundle $\cF$ on $S$ as in Proposition~\ref{rigsuk} (it exists by Lemma~\ref{champollion}).
Let  $\cE:=\cF[2]^{+}$ and let $h^{+}$ be as in~\eqref{accapiumeno}.  Lastly let
\begin{equation}\label{poler}
h:=ih^{+}.
\end{equation}
\begin{prp}\label{rosetta}
Keeping  notation  and hypotheses as above, the following hold:
\begin{enumerate}
\item
 $h$ is a primitive cohomology class, $q(h)=e$ and
$q(h,H^2(X;\ZZ))=(i)$,
\item
$r(\cE)$, $c_1(\cE)$ and $\Delta(\cE)$ are given by~\eqref{ele}.
\end{enumerate}
\end{prp}
\begin{proof}
Item~(1) holds because $\cl(D)$ is primitive, and  by Equation~\eqref{reldiqu}. 
Item~(2) follows from  Equations~\eqref{parzero}--\eqref{pardue} and straightforward computations.
\end{proof}
\begin{rmk}
In Proposition~\ref{rosetta} we have set  $\cE:=\cF[2]^{+}$. One gets an analogous result letting $\cE=\cF[2]^{-}$ (one needs to  replace
$\frac{(r_0-1)^2}{4}$   by $\frac{(r_0+1)^2}{4}$ in~\eqref{emmezero}).
\end{rmk}
\subsection{Restriction of  $\cF[2]^{\pm}$ to Lagrangian fibers}\label{subsec:restringo}
\setcounter{equation}{0}
\begin{dfn}\label{laghilb}
If $S\to\PP^1$ is an elliptically fibered $K3$ surface,  the \emph{associated Lagrangian fibration} is the   composition 
\begin{equation}
S^{[2]}\to S^{(2)}\to (\PP^1)^{(2)}\cong \PP^2.
\end{equation}
\end{dfn}
In the present subsection we  prove the following result.
\begin{prp}\label{acca20}
Let $S$ be a $K3$ surface with an elliptic fibration $S\to \PP^1$ as in Claim~\ref{eccoell}, and let  
$\pi\colon S^{[2]}\to\PP^2$ be the associated Lagrangian fibration. Let $\cF$ be a vector bundle on $S$ as in Proposition~\ref{rigsuk}. Then the restriction of $\cF[2]^{\pm}$ to every (scheme theoretic) fiber of $\pi$ is simple.
\end{prp}
The proof is given at the end of the  subsection. 

Let $S$ and  $\pi\colon S^{[2]}\to  \PP^2$ be as in Proposition~\ref{acca20}. We describe the fibers of $\pi$. 

For $x\in\PP^1$ we let $C_x$ be the fiber  over $x$ of the elliptic fibration $S\to\PP^1$. The set theoretic fibers of $\pi$ are as follows:
\begin{equation}
\pi^{-1}(x_1+x_2)=
\begin{cases}
C_{x_1}\times C_{x_2} & \text{if $x_1\not= x_2$,} \\
C_x^{(2)}\cup\PP(\Theta_{S|C_x})  & \text{if $x_1= x_2=x$.} 
\end{cases}
\end{equation}
As is easily checked the fiber is reduced if $x_1\not= x_2$. On the other hand,  the fiber over $2x$ is not reduced. In order to prove it, we introduce some notation.
Let  $V^{[2]}\subset S^{[2]}$ be the prime divisor parametrizing vertical 
subschemes $Z$ (i.e.~such that the scheme theoretic image $f(Z)$ is a reduced point), let $\Delta^{[2]}\subset S^{[2]}$ be the prime divisor parametrizing non reduced subschemes,   and let $D^{(2)}\subset (\PP^1)^{(2)}$ be the prime divisor parametrizing multiple $0$-cycles $2x$. 
\begin{prp}\label{duevu}
Let $S$ be a $K3$ surface with an elliptic fibration $S\to \PP^1$  as in Claim~\ref{eccoell}. With notation as above, we have the equality of (Cartier) divisors $\pi^{*}(D^{(2)})=2V^{[2]}+ \Delta^{[2]}$.
\end{prp}
\begin{proof}
By the set theoretic equality $\pi^{-1}(D^{(2)})=V^{[2]}\cup \Delta^{[2]}$, there exists positive integers $a,b$ such that 
 $\pi^{*}(D^{(2)})=aV^{[2]}+ b \Delta^{[2]}$. We have a commutative square 
\begin{equation}
\xymatrix{
\wt{S\times S}  \ar@{->}^{p}[r]  \ar@{->}^{\wt{\pi}}[d] & S^{[2]} \ar@{->}^{\pi}[d] \\
\PP^1\times\PP^1    \ar@{->}^{\ov{p}}[r]  & \PP^2.
}
\end{equation}
Let $\wt{V}^{[2]}:=p^{-1}(V^{[2]})$, and let $\text{Diag}_{\PP^1}$ be the diagonal of $\PP^1\times\PP^1$. The proposition follows from the equalities
\begin{equation*}
a\wt{V}^{[2]}+ 2b E=p^{*}(\pi^{*}(D^{(2)}))=\wt{\pi}^{*}(\ov{p}^{*}(D^{(2)}))=\wt{\pi}^{*}(2\text{Diag}_{\PP^1})=2\wt{V}^{[2]}+ 2 E.\qedhere
\end{equation*}
\end{proof}
\begin{crl}\label{scered}
Keep hypotheses and notation as in Proposition~\ref{duevu}. Let $\cL$ be the line bundle on $S^{[2]}$ such that 
$c_1(\cL)=\delta$. Let $x\in \PP^1$, and let $Z_x=\pi^{-1}(2x)$ be the schematic fiber of $\pi$ over $2x$. Then
\begin{equation}\label{idered}
\cI_{Z_x^{\text{red}}/Z_x |C_x^{(2)}}\cong \cO_{C_x^{(2)}}(-\Xi_x)\otimes\left(\cL_{|C_x^{(2)}}\right),
\end{equation}
where $\Xi_x:=\{2p\mid p\in C_x\}\subset C_x^{(2)}$. In particular the restriction homomorphism $\cO_{Z_x}\to 
\cO_{Z_x^{\text{red}}}$ fits into the exact sequence
\begin{equation}\label{sucstand}
0 \lra \cO_{C_x^{(2)}}(-\Xi_x)\otimes\left(\cL_{|C_x^{(2)}}\right) \lra \cO_{Z_x} \lra \cO_{Z_x^{\text{red}}}  \lra 0.
\end{equation}
\end{crl}
\begin{proof}
It suffices to prove that~\eqref{idered} holds, because the kernel of the surjection $\cO_{Z_x} \lra \cO_{Z_x^{\text{red}}}$ is equal to the left hand side in~\eqref{idered}. By Proposition~\ref{duevu} we have 
\begin{equation}\label{cruciani}
\cI_{Z_x^{\text{red}}/Z_x |C_x^{(2)}}\cong \cO_{S^{[2]}}(-V^{[2]}-\Delta^{[2]})_{|C_x^{(2)}}.
\end{equation}
We have
\begin{equation}\label{delatore}
\cO_{S^{[2]}}(-\Delta^{[2]})_{|C_x^{(2)}}\cong \cO_{C_x^{(2)}}(-\Xi_x).
\end{equation}
On the other hand Proposition~\ref{duevu} gives that
\begin{equation*}
\cO_{S^{[2]}}(2V^{[2]})\cong \cO_{S^{[2]}}(\pi^{*}D^{(2)}-\Delta^{[2]})\cong \pi^{*}(\cO_{\PP^2}(2))\otimes\cL^{-2}.
\end{equation*}
Since $H^1(S^{[2]};\ZZ)=0$, it follows that 
\begin{equation*}
\cO_{S^{[2]}}(V^{[2]})\cong  \pi^{*}(\cO_{\PP^2}(1))\otimes\cL^{-1}.
\end{equation*}
 Since the restriction of $\mu(c_1(\cO_{\PP^2}(1)))$ to $C_x^{(2)}$ is trivial, we get that
\begin{equation}\label{vertigo}
\cO_{S^{[2]}}(-V^{[2]})_{|C_x^{(2)}}\cong  \cL_{|C_x^{(2)}}.
\end{equation}
Hence~\eqref{idered} follows from~\eqref{cruciani}, \eqref{delatore} and~\eqref{vertigo}.
\end{proof}
For $x\in\PP^1$  let $\cF_x:=\cF_{|C_x}$. If $x_1\not= x_2\in\PP^1$, then the restriction of $\cF$ to $\pi^{-1}(x_1+x_2)$ is equal to 
$\cF_{x_1}\boxtimes\cF_{x_2}$. Thus we need the following result.
\begin{prp}\label{quadstab}
Let $C_i$ be  polarized irreducible curves, and  let $D$ be an ample divisor on $Y:=C_1\times C_2$ such that
\begin{equation}\label{tetamenne}
c_1(\cO_Y(D))=m_1 \rho_1^{*} \cO_{Y_1}(p_1)+ m_2 \rho_2^{*} c_1(\cO_{Y_2}(p_2)),
\end{equation}
where $\rho_i\colon Y\to C_i$ is the projection, and $p_i\in C_i$. Let $\cV_i$ be a slope-stable vector bundle on 
$C_i$ for $i\in\{1,2\}$. 
Then $\cV_1\boxtimes\cV_2$ is $D$ slope-stable. 
\end{prp}
\begin{proof}
By contradiction. Suppose the contrary. Hence exists an injection $\alpha\colon\cE\to \cV_1\boxtimes\cV_2$  with  torsion free cokernel such that $0<r(\cE)<r(\cV_1\boxtimes\cV_2)$ and
\begin{equation}\label{destabo}
\mu_{D}(\cE)\ge \mu_{D}(\cV_1\boxtimes\cV_2).
\end{equation}
Let  $p_i\in C_i$  be generic; the restrictions of $\alpha$ to $\{p_1\}\times C_2$ and  $C_1\times \{p_2\}$    are injective maps of vector bundles (that is, injective on fibers). We have
\begin{equation}\label{sanmartino}
\mu_{D}(\cE)=m_2\mu(\cE_{|{C_1}\times \{p_2\}})+m_1 \mu(\cE_{|\{p_1\}\times C_2}).
\end{equation}
On the other hand 
\begin{equation}\label{pendab}
\mu_{D}(\cV_1\boxtimes\cV_2)=m_2 \mu(\cV_1)+m_1\mu(\cV_2).  
\end{equation}
Since the restrictions of $\cV_1\boxtimes\cV_2$  to $C_1\times \{p_2\}$ and $\{p_1\}\times C_2$ are isomorphic to the polystable vector bundles   $\cV_1\otimes_{\CC}\CC^{r(\cV_2)}$ and $\cV_1\otimes_{\CC}\CC^{r(\cV_1)}$ respectively,  it follows from~\eqref{destabo}, \eqref{sanmartino} and~\eqref{pendab}  that $\mu(\cE_{|C_1\times \{p_2\}})=\mu(\cV_1)$ and $\mu(\cE_{|\{p_1\}\times C_2})=\mu(\cV_2)$. In turn, these equalities give that there exist vector subspaces 
$0\not=A\subset \CC^{r(\cV_2)}$ and $0\not=B\subset \CC^{r(\cV_1)}$ such that 
$\cE_{|C_1\times \{p_2\}}=\cV_1\otimes_{\CC} A$ and $\cE_{|\{p_1\}\times C_2}=B\otimes_{\CC}\cV_2$ respectively. It follows that $\im\alpha=\cV_1\boxtimes\cV_2$. This is a contradiction.
\end{proof}

\begin{prp}\label{semgen}
Let $S$ be a $K3$ surface with an elliptic fibration $S\to \PP^1$  as in Claim~\ref{eccoell}, and  let $\cF$ be a vector bundle on $S$ as in Proposition~\ref{rigsuk}. 
If $x_1\not= x_2$
the restriction of $\cF[2]^{\pm}$ to  $\pi^{-1}(x_1+x_2)=C_{x_1}\times C_{x_2}$ is slope-stable for any product polarization, and it is  a simple semi-homogeneous  vector bundle.
\end{prp}
\begin{proof}
We have
\begin{equation}\label{scatper}
\cF[2]^{\pm}_{|\pi^{*}(x_1+x_2)}\cong \cF_{x_1}\boxtimes \cF_{x_2}. 
\end{equation}
By Proposition~\ref{rigsuk}, both $\cF_{x_1}$ and $\cF_{x_2}$ are slope-stable. Hence the statement about slope-stability follows from Proposition~\ref{quadstab}. Of course this implies that the restriction  of $\cF[2]^{\pm}$  to  
$\pi^{-1}(x_1+x_2)=C_{x_1}\times C_{x_2}$ is simple. Moreover it is semi-homogeneous because every stable vector bundle on an elliptic curve is semi-homogeneous. (One could argue that 
 the discriminant  vanishes by  Lemma~\ref{zerene}, and conclude by Proposition~\ref{discperp}.)
\end{proof}
For $x\in\PP^1$ let $\Delta^{[2]}_x:=\PP(\Theta_{S|C_x})\subset\Delta^{[2]}$, and we let $\cO_{\Delta^{[2]}_x}(1)$ be the dual of the tautological line subbundle $\Delta^{[2]}_x$.
Let  $\tau_x\colon\Delta^{[2]}_x\to C_x$ be the restriction of the  Hilbert-Chow map. 
\begin{lmm}\label{somdir}
Let $S$ be a $K3$ surface with an elliptic fibration $S\to \PP^1$, and  let $\cF$ be a vector bundle on $S$ as in Proposition~\ref{rigsuk}. Let $x$ be a regular value of the elliptic fibration. Then 
\begin{eqnarray}\label{spezza}
\cF[2]^{+}_{|\Delta^{[2]}_x} &  \cong  & 
\left(\cO_{\Delta^{[2]}_x}(1) \otimes   \tau_{x}^{*}\bigwedge^2\cF_x  \right) \oplus \tau_{x}^{*}(\Sym^2\cF_x), \\
\cF[2]^{-}_{|\Delta^{[2]}_x} &  \cong  & 
\left(\cO_{\Delta^{[2]}_x}(1) \otimes   \tau_{x}^{*}\Sym^2\cF_x  \right) \oplus \tau_{x}^{*}\left(\bigwedge^2\cF_x\right),
\end{eqnarray}
and the space of traceless endomorphisms of
 $\cF[2]^{\pm}_{|\Delta^{[2]}_x}$ has dimension $1$.
\end{lmm}
\begin{proof}
The proofs  for $\cF[2]^{\pm}$ are similar. We provide the proof  for $\cF[2]^{+}$.
 Restricting the defining exact sequence~\eqref{eccopi} to $E$, we get the exact sequence
\begin{equation}\label{resdel}
0\lra \tau_{E}^{*}(\bigwedge^2\cF)\otimes \cO_{E}(-E) \lra \cF[2]^{+}_{|E}\lra  \tau_{E}^{*}(\Sym^2\cF)\lra 0.
\end{equation}
 Restricting to $\Delta^{[2]}_x$ the exact sequence in~\eqref{resdel}, we get the exact sequence
\begin{equation}\label{resixe}
0\lra \tau_{x}^{*}( \bigwedge^2\cF_x )\otimes \cO_{\Delta^{[2]}_x}(1) \lra \cF[2]^{+}_{|\Delta^{[2]}_x} \overset{\beta}{\lra}  
\tau_{x}^{*}(  \Sym^2\cF_x   )\lra 0.
\end{equation}
We claim that the above exact sequence splits. The extension class belongs to 
\begin{equation}\label{claest}
H^1(\Delta^{[2]}_x, \tau_{x}^{*}(\Sym^2\cF^{\vee}_x\otimes \bigwedge^2\cF_x)\otimes \cO_{\Delta^{[2]}_x}(1)).
\end{equation}
We compute the above cohomology group via the Leray spectral sequence  of $\tau_x$. Since  
$(\Sym^2\cF^{\vee}_x\otimes \bigwedge^2\cF_x)\otimes R^1\tau_{x,*} \cO_{\Delta^{[2]}_x}(1)=0$,  
  it suffices to show that
\begin{multline}\label{singapore}
H^1(C_x, (\Sym^2\cF^{\vee}_x\otimes \bigwedge^2\cF_x)\otimes\tau_{x,*} \cO_{\Delta^{[2]}_x}(1))=\\
=H^1(C_x, (\Sym^2\cF^{\vee}_x\otimes \bigwedge^2\cF_x)\otimes \Theta^{\vee}_{S|{C_x}})=0.
\end{multline}
Since $C_x$ is smooth, the  sheaf $\Theta^{\vee}_{S|{C_x}}$ has a filtration with associated graded the direct sum of two copies of $\cO_{C_x}$.  Hence the vector bundle appearing in~\eqref{singapore} has a filtration whose  associated graded is  the direct sum of two copies of 
$\Sym^2\cF^{\vee}_x\otimes \bigwedge^2\cF_x$. The latter vector bundle has no non zero section because it is slope-stable of slope $0$. By HRR  it follows that
\begin{multline*}\label{singapore}
H^0(C_x, (\Sym^2\cF^{\vee}_x\otimes \bigwedge^2\cF_x)\otimes \Theta^{\vee}_{S|{C_x}})=\\
=H^1(C_x, (\Sym^2\cF^{\vee}_x\otimes \bigwedge^2\cF_x)\otimes \Theta^{\vee}_{S|{C_x}})=0.
\end{multline*}
Hence the extension group in~\eqref{claest} vanishes.

The result about traceless endomorphisms of $\cF[2]^{\pm}_{|\Delta^{[2]}_x}$ follows from the direct sum decomposition in~\eqref{spezza} and the vanishing in~\eqref{singapore}.
\end{proof}
\begin{lmm}\label{cidue}
Let $S$ be a $K3$ surface with an elliptic fibration $S\to \PP^1$  as in Claim~\ref{eccoell}, and  let $\cF$ be a vector bundle on $S$ as in Proposition~\ref{rigsuk}. Let $x$ be a regular value of the elliptic fibration. Let $p_x\colon C_x^2\to S^{[2]}$ be composition of the quotient map $C_x^2\to C_x^{(2)}$ and the inclusion $C_x^{(2)}\hra S^{[2]}$. Then  $p_x^{*}\cF[2]^{\pm}$ is simple, and hence 
 the restriction of
$\cF[2]^{\pm}$ to $C_x^{(2)}$ is simple.
\end{lmm}
\begin{proof}
The proofs for $\cF[2]^{\pm}$ are similar. We give the proof for $\cF[2]^{+}$.  

Let $q_x,r_x\colon C_x^2\to C_x$ be the two projections, and let $\Diag_x\subset C_x^2$ be the diagonal.
By the exact sequence in~\eqref{eqcK}, we have the following exact sequence
 \begin{equation}\label{etciu}
0\lra p_x^*\cF[2]^{+}\lra q_x^*\cF_x\otimes r_x^*\cF_x \overset{\ev_x^-}{\lra}  \bigwedge^2\cF_{x|\Diag_x}\lra 0.
\end{equation}
It follows that we have the following exact sequence
 \begin{equation}\label{traslo}
0\lra  q_x^*\cF_x\otimes r_x^*\cF_x\otimes\cO_{C_x^2}(-\Diag_x)  \overset{\lambda}{\lra} p_x^*\cF[2]^{+}
\overset{\ev_x^{+}}{\lra} \Sym^2\cF_{x|\Diag_x}\lra 0.
\end{equation}
Notice that the restriction of $\lambda$ to $\Diag_x$ is identified with the natural map 
$\ev_x^{-}\colon \cF_x\otimes\cF_x\to \bigwedge^2\cF_x$ (notice that the normal bundle of $\Diag_x$ in $C_x^2$ is trivial because $C_x$ is an elliptic curve), in particular its image is $\bigwedge^2\cF_x$ (we identify $\Diag_x$ and $C_x$). Equivalently, the restriction to $\Diag_x$ 
of the exact sequence in~\eqref{traslo} gives rise to the exact sequence
 \begin{equation}\label{traslo2}
0\lra   \bigwedge^2\cF_x \lra p_x^*\cF[2]^{+}_{|\Diag_x}
\overset{\ev_x^{+}}{\lra}    \Sym^2\cF_x   \lra 0,
\end{equation}
which is split by Lemma~\ref{somdir}. Now let  $\varphi$ be an endomorphism of  $p_x^*\cF[2]^{+}$. The restriction of $\varphi$ to $\Diag_x$ preserves the exact sequence in~\eqref{traslo2} because the vector bundles $\Sym^2\cF_x$,  $\bigwedge^2\cF_x$ are slope-stable of equal slopes.  It follows that  $\varphi$ induces an endomorphism of the kernel of $\ev_x^{+}$, i.e.~of 
$q_x^*\cF_x\otimes r_x^*\cF_x\otimes\cO_{C_x^2}(-\Diag_x)$. Since $\cF_x$ is simple, the latter is a simple sheaf. It follows that $\varphi$ is a scalar.
\end{proof}
\begin{prp}\label{sempdueix}
Let $S$ be a $K3$ surface with an elliptic fibration $S\to \PP^1$  as in Claim~\ref{eccoell}, and  let $\cF$ be a vector bundle on $S$ as in Proposition~\ref{rigsuk}. Let $x$ be a regular value of the elliptic fibration. Then
the restriction of $\cF[2]^{\pm}$ to the scheme theoretic fiber $\pi^{-1}(2x)$ is a simple sheaf.
\end{prp}
\begin{proof}
Let $Z_x:=\pi^{-1}(2x)$ be  the scheme theoretic fiber.  By Corollary~\ref{scered} we have an exact sequence
\begin{equation}
0\lra  End^0(\cF)_{|C_x^{(2)}}\otimes\cO_{C_x^{(2)}}(-\Xi_x)\otimes\left(\cL_{|C_x^{(2)}}\right) \lra  End^0(\cF)_{|Z_x}  \lra  End^0(\cF)_{|Z^{\text{red}}_x} \lra 0.
\end{equation}
 Taking global sections, we get an isomorphism
\begin{equation}
H^0(C_x^{(2)},  End^0(\cF)_{|C_x^{(2)}}\otimes\cO_{C_x^{(2)}}(-\Xi_x)\otimes\left(\cL_{|C_x^{(2)}}\right))\overset{\sim}{\lra} 
H^0(C_x^{(2)},  End^0(\cF)_{|Z_x}).
\end{equation}
because of  Lemmas~\ref{somdir} and~\ref{cidue}. On the other hand, since 
\begin{equation*}
p_x^{*}\left(\cO_{C_x^{(2)}}(-\Xi_x)\otimes\left(\cL_{|C_x^{(2)}}\right)\right)\cong \cO_{C_x^2}(-\Diag_x), 
\end{equation*}
we have an embedding
\begin{equation}
H^0(C_x^{(2)},  End^0(\cF)_{|C_x^{(2)}}\otimes\cO_{C_x^{(2)}}(-\Xi_x)\otimes\left(\cL_{|C_x^{(2)}}\right) \hra
H^0(C_x^2,  End^0(p_x^{*}\cF_x)(-\Diag_x)),
\end{equation}
and the latter vanishes by Lemma~\ref{cidue}.
\end{proof}
\begin{proof}[Proof of  Proposition~\ref{acca20}]
 Let $B\subset\PP^2$ be the (finite) set parametrizing $2x\in (\PP^1)^{(2)}$, where $x$ is a critical value of the elliptic fibration $S\to\PP^1$.  Let us prove that  if
  $(x_1+x_2)\in(\PP^2\setminus B)$ then
\begin{equation}\label{tuttinulli}
h^p(\pi^{-1}(x_1+x_2), End^0 \cF[2]^{\pm}_{|\pi^{-1}(x_1+x_2)})=0 \quad \forall p. 
\end{equation}
 To see this, first notice that since $\pi^{-1}(x_1+x_2)$ is a local complete intersection with trivial dualizing sheaf, and $\chi(\pi^{-1}(x_1+x_2), End^0 \cF[2]^{\pm}_{0|\pi^{-1}(x_1+x_2)})=0$, it suffices to check that~\eqref{tuttinulli} holds for $p=0$. 
 
 If $x_1\not=x_2$, then~\eqref{tuttinulli} holds for $p=0$ by Proposition~\ref{semgen}. If $x_1=x_2=x$, and $x$ is a regular value of the elliptic fibration, 
then~\eqref{tuttinulli} holds for $p=0$ by Proposition~\ref{sempdueix}. This proves that~\eqref{tuttinulli} holds for all $(x_1+x_2)\in(\PP^2\setminus B)$. Since  $B$ is a finite set we get that
\begin{equation}\label{imdirzero}
R^p\pi_{*}End^0 \cF[2]^{\pm}=0, \quad p\in\{0,1\}.
\end{equation}
(See Proposition 2.26 in~\cite{mukvb}.) Now suppose that  there exist $x_1+x_2\in\PP^2$
such that  the restriction of $\cF[2]^{\pm}$ to $\pi^{-1}(x_1+x_2)$ is not simple. As shown above such points are contained in the finite set $B$, and hence it follows
 (since the fibers of $\pi$ are surfaces) that 
$R^2\pi_{*}End^0 (\cF[2]^{\pm})$ is a non zero Artinian sheaf.  By the Leray spectral sequence for $\pi$, and the vanishing in~\eqref{imdirzero},  it follows that $H^2(S^{[2]},End^0 (\cF[2]^{\pm}))\not=0$. This contradicts Proposition~\ref{prp:seirigido}.
\end{proof}

 \section{Proof of Theorem~\ref{unicita} and Corollary~\ref{solosolo}}\label{ledimo}
\subsection{Summary}
\setcounter{equation}{0}
In  Subsection~\ref{baggio} we prove that  if $d\gg 0$ there exists an irreducible component 
$\cN^i_e(d)^{\text {good}}\subset\cN^i_e(d)$, where $\cN^i_e(d)\subset\cK^i_e$ is the Lagrangian Noether-Lefschetz divisor defined in Definition~\ref{ennelagr}, with the following property: 
if $[(X,h)]\in\cN^i_e(d)^{\text {good}}$ is generic,
 there exists an $h$ slope-stable vector bundle $\cE$ on $X$ with   $\ch_0(\cE),\ch_1(\cE),\ch_2(\cE)$ given by~\eqref{ele} 
  whose restriction to Lagrangian fibers is slope-stable  with the possible exception of a finite set of fibers. We also prove that there exists one  component of the relative moduli space of slope-stable vector bundles on polarized HK's parametrized by $\cK^i_e$ with 
   $\ch_0,\ch_1,\ch_2$ given by~\eqref{ele}  which dominates the moduli space $\cK^i_e$.  
  
 In Subsection~\ref{delpiero} we prove that if $[(X,h)]\in\cN^i_e(d)^{\text {good}}$ is as above, there is a single $h$ slope-stable vector bundle  with the relevant Chern character.  

In  Subsection~\ref{dimfin} we will prove Theorem~\ref{unicita} and  Corollary~\ref{solosolo}. 
\subsection{Good vector bundles over Lagrangian HK's}\label{vblagruno}\label{baggio}
\setcounter{equation}{0}

Below is the first main result of the present subsection.
\begin{prp}\label{buonacompt}
Let $i\in\{1,2\}$ and let $r_0\ge 2$ such that $i\equiv r_0\pmod{2}$. Suppose  that~\eqref{econ}  holds, that $e\notdivides 2d$ and that 
\begin{equation}\label{pazienza}
d>\frac{5}{16}r_0^6(r_0^2-1)(e+1).
\end{equation}
There exists an irreducible component 
$\cN^i_e(d)^{\text {good}}\subset\cN^i_e(d)$, where $\cN^i_e(d)\subset\cK^i_e$  is as in Definition~\ref{ennelagr}, such that the following holds.
Let $[(X,h)]\in \cN_e^i(d)^{\text {good}}$ be  generic. (Notice that  the hypotheses of Proposition~\ref{unicafibr} hold, and hence there is an associated Lagrangian fibration 
$\pi\colon X\to\PP^2$.)
Then 
\begin{enumerate}
\item
there exists  an $h$ slope-stable vector bundle $\cE$ on $X$ such that~\eqref{ele} holds, and
\item
except possibly for a finite set of $z\in\PP^2$,    the restriction of $\cE$ to $\pi^{-1}(z)$ 
  is    slope-stable for the restricted polarization.
\end{enumerate}
\end{prp}
The proof of Proposition~\ref{buonacompt} is given at the end of the subsection.

Let $S$ be an elliptic $K3$ surface as in Subsection~\ref{traduzione}, and let us adopt the notation of that subsection. Let $X_0=S^{[2]}$, and let $\cE_0:=\cF[2]^{+}$ be the vector bundle on $X_0$ of loc.~cit. Let 
 $h_0:=h$, where $h$ is given by~\eqref{poler}. Let $C\subset S$ be a fiber of the elliptic fibration and let 
$f_0:=\mu(\cl(C))$. Lastly let $d_0$ be as in~\eqref{neronsevero} and set
\begin{equation}
d:=i d_0.
\end{equation}
Then the sublattice $\la f_0,h_0\ra\subset H^{1,1}_{\ZZ}(X_0)$ is saturated and
\begin{equation}\label{edizero}
q(f_0)=0,\quad q(h_0,f_0)=d,\quad q(h_0)=e.
\end{equation}
Let $\pi_0\colon X_0\to\PP^2$ be the Lagrangian fibration associated to the elliptic fibration of $S$, see Definition~\ref{laghilb}. Notice that 
$f_0=c_1(\pi_0^{*}\cO_{\PP^2}(1))$. 

Let $\varphi\colon\cX\to B$ be an analytic representative of the deformations space of $(X,\la h_0,f_0\ra)$ i.e.~deformations of $X_0$ that keep $h_0$ and $f_0$ of Hodge type. 
We assume that $B$ is contractible. Let  $0\in B$ the base point, in particular $X_0$ is isomorphic to $\varphi^{-1}(0)$. For $b\in B$ we let $X_b:=\varphi^{-1}(b)$. 
If $B$ is small enough, then by  Proposition~\ref{acca20} and Corollary~\ref{iacman} the vector bundle $\cE_0$ on $X_0$ deforms to a vector bundle $\cE_b$ on $X_b$ (unique up to isomorphism because $H^1(X_0,End^0(\cE_0))=0$). Notice that $\la h_0,f_0\ra$ deforms  by Gauss-Manin parallel transport to a saturated sublattice 
\begin{equation}\label{lambdabi}
\Lambda_b:=\la h_b,f_b\ra\subset H^{1,1}_{\ZZ}(X_b).
\end{equation}
Possibly after shrinking $B$ around $0$ there exists a map $\pi\colon\cX\to \PP^2$ which  restricts to  a Lagrangian fibration on every $X_b$, and is equal to  $\pi_0$  on $X_0$. We let $\pi_b$ be the restriction of $\pi$ to $X_b$. 
Notice that the fiber of $\varphi\times\pi\colon\cX\to B\times\PP^2$ over 
$(b,z)$ is the Lagrangian fiber of $X_b\to\PP^2$ over $z$; we denote it by $X_{b,z}$. Of course $f_b=c_1(\pi_b^{*}\cO_{\PP^2}(1))$.

\begin{prp}\label{chiave}
With the hypotheses of Proposition~\ref{buonacompt}, the following holds.  For $b\in B$ outside a proper analytic subset  
$h_b$ is ample and hence $[(X_b,h_b)]\in\cN^i_e(d)$ where $i\equiv r_0\pmod{2}$. 
Moreover $\cE_b$ is $h_b$ slope-stable and Items~(1), (2) of Proposition~\ref{buonacompt} hold for $\cE=\cE_b$.
\end{prp}
\begin{proof}
Since $r_0\ge 2$, Inequality~\eqref{pazienza} implies that $d>10(e+1)$. Hence the hypotheses of Proposition~\ref{unicafibr} hold. In the proof of that proposition we showed that $h_b$ is ample for $b$ outside a proper analytic subset  of $B$.

Next we prove that $h_b$ is $a(\cE_b)$-suitable, see Definition~\ref{suipol}. We claim that the hypotheses of Proposition~\ref{propriostab} hold  with $a_0=a(\cE_b)$. This is clear for all the hypotheses, except perhaps for the inequality in~\eqref{golfangora}.  
By Proposition~\ref{yaufever} we have $a(\cE_b)=\frac{5}{8} r_0^6(r_0^2-1)$, and hence the inequality in~\eqref{golfangora} follows from~\eqref{pazienza}.  

Since $h_b$ is $a(\cE_b)$-suitable, in order to prove that $\cE_b$ is $h_b$ slope-stable it suffices to show that the restriction to a generic fiber of the Lagrangian fibration is slope-stable, see Proposition~\ref{lagstab}. This is true for $\cE_0$ by Proposition~\ref{semgen}. By openness of slope-stability, it follows that it is true also for $b\in B$ outside a proper analytic subset.

Next we prove that Items~(1) and (2) of Proposition~\ref{buonacompt} hold for $\cE=\cE_b$. 

Item~(1) holds by Proposition~\ref{rosetta}.   

Let us  prove that Item~(2) holds. First we notice that for $b\in B$ outside a proper closed analytic subset  the restriction of $\cE_b$ to 
every Lagrangian fiber is simple. In fact this holds for $b=0$ by Proposition~\ref{acca20}, and hence the assertion we made holds by openness of \lq\lq simpleness\rq\rq. Let us prove that for $b\in B$ outside a proper closed analytic subset  the restriction of $\cE_b$ to 
a smooth Lagrangian fiber is slope-stable. By Proposition~\ref{semgen} the restriction of $\cE_0$ to a generic Lagrangian fiber  is slope-stable, and hence the restriction of 
$\cE_b$ (for $b\in B$ outside...) to a generic Lagrangian fiber  is slope-stable by openness of slope-stability. By Proposition~\ref{discperp} we get that 
the restriction of $\cE_b$ (for $b\in B$ outside...) to a generic smooth Lagrangian fiber  is semi-homogeneous. It follows that the restriction to any smooth Lagrangian fiber  is simple semi-homogeneous (note: the fact that the restriction is simple 
is crucial). By Proposition~6.16 in~\cite{muksemi} the restriction of $\cE_b$  to any smooth Lagrangian fiber is Gieseker-Maruyama stable, and hence slope-semistable. By Corollary~\ref{caravilla} it follows that it is actually  slope-stable.

Next we  claim  that 
the restriction of $\cE_b$ (for $b\in B$ outside...) to a generic singular Lagrangian fiber  is slope-stable, except possibly for a finite set of fibers. The singular Lagrangian fibers of $X_b$ are parametrized by the discriminant curve of $\pi_b$, which, for  $b\in B$ outside a proper closed analytic subset, is the dual curve of a generic sextic plane curve, see Proposition~\ref{torcere}.   
On the other hand the discriminant curve of $\pi_0$ is the union of $24$ lines (each corresponding to a critical value of the elliptic fibration) and a conic (the \lq\lq diagonal\rq\rq). The restriction of $\cE_0$ to the  Lagrangian surface parametrized by a generic point of one of the lines is slope-stable by Proposition~\ref{semgen}. By openness of slope-stability, it follows that the locus of singular Lagrangian fibers on which $\cE_b$ restricts to a slope-stable vector bundle is non empty (for $b\in B$ outside...). Since (for $b\in B$ outside...) the discriminant curve is irreducible, this proves that, with the possible exception of a finite set of singular fibers, the restriction of $\cE_b$ (for $b\in B$ outside...) to a generic singular Lagrangian fiber  is slope-stable.

\end{proof}
\begin{proof}[Proof of Proposition~\ref{buonacompt}]
Let $\cX\to T_{e}^1$ and $\cX\to T_{e}^2$ be  complete families of  polarized 
 HK's of Type $K3^{[2]}$ such that~\eqref{divuno}, respectively ~\eqref{divdue}, holds - e.g.~the families parametrized by the relevant  open subsets of  suitable Hilbert schemes. We may, and will, assume that $T_{e}^i$ is irreducible.  
 For $t\in T_{e}^i$ we let $(X_t,h_t)$ be the corresponding polarized HK of Type $K3^{[2]}$. Let $m\colon T_{e}^i\to \cK^i_e$ be the moduli map, which sends $t$ to $[(X_t,h_t)]$.

By fundamental results of Gieseker and Maruyama  there exists a  map of schemes
\begin{equation}\label{vbrel}
f\colon \cM_{e}(r_0)\to T^i_e 
\end{equation}
 such that for every $t\in T^i_e$ the (scheme theoretic) fiber $f^{-1}(t)$  is isomorphic to the (coarse) moduli space  of 
  $h_t$ slope-stable vector bundles $\cE$ on $X_t$ such that~\eqref{ele} holds. Moreover  $f\colon \cM_{e}(r_0)\to T^i_e$ is  of finite type by Maruyama~\cite{marubound}, and hence  $f(\cM_{e}(r_0))$ is a  constructible subset of $T^i_e$.
  
  By Proposition~\ref{chiave}, the image of 
  $f\colon \cM_{e}(r_0)\to T^i_e$ contains a non empty subset of $m^{-1}(\cN^i_e(d))$ which is open in the classical topology. Since the image of $f$ is a constructible set, it follows that it contains a Zariski open dense subset of an irreducible component of $m^{-1}(\cN^i_e(d))$. Since the image of $f$ is the inverse image of a subset of the moduli map $m$, it follows that there exists an irreducible component $\cN^i_e(d)^{\text{good}}\subset \cN^i_e(d)$ such that the image of $f$ contains a Zariski open dense subset of  $m^{-1}(\cN^i_e(d)^{\text{good}})$.

\end{proof}
Next we state the second main result of the present subsection. Let us agree that a map of quasi-projective varieties is dominant if the image is Zariski-dense in the codomain (usually the attribute dominant is reserved to maps between irreducible varieties). 
\begin{prp}\label{lastminute}
With notation as above, the map $f\colon \cM_{e}(r_0)\to T^i_e$ is dominant.
\end{prp}
\begin{proof}
Follows at once from Corollary~\ref{iacman}.
\end{proof}
\subsection{Unicity of stable vector bundles on lagrangian HK's}\label{delpiero}
\setcounter{equation}{0}
We will prove the result below.
\begin{prp}\label{unico}
Let $i\in\{1,2\}$. Suppose that $r_0\ge 2$, that $r_0\equiv i\pmod{2}$,  that~\eqref{econ}  holds, that $e\notdivides 2d$ and that 
\begin{equation}\label{santapi}
d>\frac{5}{16}r_0^6(r_0^2-1)(e+1).
\end{equation}
Let $[(X,H)]\in \cN_e^i(d)^{\text{good}}$ be generic, where $\cN_e^i(d)^{\text{good}}$ is as in Proposition~\ref{buonacompt}. Then, up to isomorphism,  there exists  one and only one $h$ slope-stable vector bundle 
$\cE$ on $X$ such that~\eqref{ele} holds. 
\end{prp}
We first prove the following auxiliary result.
\begin{lmm}\label{oglemma}
Let $(Y,h)$ be a polarized irreducible smooth projective variety, and let $\rho\colon Y\to T$ be a dominant map to a smooth curve with integral fibers of dimension $n$. For $t\in T$ set $Y_t:=\rho^{-1}(t)$ and $h_t:=h_{|Y_t}$.
Let $\cF$ and $\cG$ be locally free sheaves on $Y$ such that the following hold:
\begin{enumerate}
\item
$\Delta(\cF)\smile h^{n-1}=\Delta(\cG)\smile h^{n-1}$.
\item
The restriction $\cF_t:=\cF_{|Y_t}$  is $h_t$ slope-stable for all $t\in T$. 
\item
The restrictions $\cF_t$ and $\cG_t:=\cG_{|Y_t}$ are isomorphic for generic $t\in T$.
\end{enumerate}
Then $\cF_t$ and $\cG_t$ are isomorphic for  \emph{all}  $t\in T$. 
\end{lmm}
\begin{proof}
The sheaf $\cL:=Hom_{\rho}(\cG,\cF)$ is torsion-free because $\cG$ and $\cF$ are locally free, and its fiber over    a generic point of $T$ is $1$ dimensional by Items~(2) and~(3). Since $T$ is a smooth curve it follows that  $\cL:=Hom_{\rho}(\cG,\cF)$ is an invertible sheaf.   The tautological map 
$\cG\otimes \rho^{*}(\cL)\to\cF$  gives rise to an exact sequence 
\begin{equation}
0\lra \cG\otimes\cL\overset{\alpha}{\lra} \cF\lra \cQ \lra 0.
\end{equation}
 It suffices to show that $\cQ=0$. Suppose that $\cQ\not=0$; we claim that
 \begin{equation}\label{deltacresce}
\int\limits_{Y}\Delta(\cG\otimes\rho^{*}(\cL))\smile h^{n-1}>
\int\limits_{Y}\Delta(\cF)\smile h^{n-1}.
\end{equation}
In order to prove this we notice that the (set theoretic) support of $\cQ$ is  a disjoint union of fibers of $\rho$, say 
$Y_{t_1},\ldots,Y_{t_d}$ and that $\alpha_t\colon \cG_t\to\cF_t$ is \emph{not} an isomorphism if and only if $t\in\{t_1,\ldots,t_d\}$. Thus we have $\cQ=\oplus_{i=1}^d\cQ_i$ where the set-theoretic support of $\cQ_i$ is 
$Y_{t_i}$. Let $r=\rk(\cF)=\rk(\cG)$. We have
 \begin{equation}\label{pincetto}
\Delta(\cG\otimes\rho^{*}(\cL))=
\Delta(\cF)+2\sum\limits_{i=1}^d\left(r\ch_2(\cQ_i)-c_1(\cF)\smile c_1(\cQ_i)\right).
\end{equation}
Let $\epsilon_i\in\cO_{T,t_i}$ be a local parameter at $t_i$. There exists $m_i>0$ such that $\epsilon_i^{m_i}\cQ_i=0$. For each $t_i$  we have a filtration 
$\cQ\supset\epsilon_i\cdot\cQ\supset\ldots\supset \epsilon_i^{m_i}\cdot\cQ=0$ and
\begin{multline}\label{tantipezzi}
r\ch_2(\cQ_i)-c_1(\cF)\smile c_1(\cQ_i)= \\
=\sum\limits_{\ell=1}^{m_i}\left(r\ch_2(\epsilon_i^{\ell}\cdot\cQ_i/\epsilon_i^{\ell+1}\cdot\cQ_i)-
c_1(\cF)\smile c_1(\epsilon_i^{\ell}\cdot\cQ_i/\epsilon_i^{\ell+1}\cdot\cQ_i)\right).
\end{multline}
The sheaf $\epsilon_i^{\ell}\cdot\cQ_i/\epsilon_i^{\ell+1}\cdot\cQ_i$ is annihilated by $\epsilon_i$, hence it is the pushforward of a sheaf on $Y_{t_i}$:
\begin{equation*}
\epsilon_i^{\ell}\cdot\cQ_i/\epsilon_i^{\ell+1}\cdot\cQ_i=i_{Y_{t_i},*}(\ov{\cQ}_{i,\ell}).
\end{equation*}
By GRR we get that 
\begin{multline*}
\ch_2(i_{Y_{t_i},*}(\ov{\cQ}_{i,\ell}))=i_{Y_{t_i},*}(c_1(\ov{\cQ}_{i,\ell})), \quad
c_1(\cF)\smile c_1(\ov{\cQ}_{i,\ell}))= i_{Y_{t_i},*}(\rk(\ov{\cQ}_{i,\ell}) i_{Y_{t_i}}^{*}c_1(\cF)).
\end{multline*}
Hence 
\begin{multline}\label{intpend}
\int_Y\left(r\ch_2(\epsilon_i^{\ell}\cdot\cQ_i/\epsilon_i^{\ell+1}\cdot\cQ_i)-
c_1(\cF)\smile c_1(\epsilon_i^{\ell}\cdot\cQ_i/\epsilon_i^{\ell+1}\cdot\cQ_i)\right)\smile h^{n-1}=\\
=\int_{Y_{t_i}}\left(r c_1(\ov{\cQ}_{i,\ell})-\rk(\ov{\cQ}_{i,\ell}) c_1(\cF_{t_i})\right)\smile h_{t_i}^{n-1}.
\end{multline}
We have surjections
\begin{equation}\label{grancasino}
\begin{matrix}
 \cF/\epsilon_i\cdot\cF & \overset{\phi_{i,\ell}}{\twoheadrightarrow} & 
\epsilon_i^{\ell}\cdot\cQ_i/\epsilon_i^{\ell+1}\cdot\cQ_i \\
 s & \mapsto & \epsilon_i^{\ell}\cdot s
\end{matrix}
\end{equation}
Notice that we may view $\phi_{i,\ell}$ as map of sheaves on $Y_{t_i}$, namely as 
$\phi_{i,\ell}\colon \cF_{t_i}\to \ov{\cQ}_{i,\ell}$.  By $h_{t_i}$ slope stability of $\cF_{t_i}$ it follows that if
\begin{equation}\label{tratra}
0<\rk(\ov{\cQ}_{i,\ell})<r=\rk(\cF)=\rk(\cG)
\end{equation}
the integral in~\eqref{intpend}  is strictly positive. 

Let us prove  that~\eqref{tratra} holds if $\ell=0$.  
The map 
$\alpha_{t_i}\colon\cG_{t_i}\to\cF_{t_i}$ 
 is non zero by hypothesis and $\ov{\cQ}_{i,0}=\coker(\alpha_{t_i})$. If $\rk_{Y_{t_i}}(\ov{\cQ}_{i,0})=r$ then 
 $\alpha_{t_i}$ vanishes at the generic point of $Y_{t_i}$ and hence   it vanishes at all points of $Y_{t_i}$, contradiction. 
 Now suppose that $\rk_{Y_{t_i}}(\ov{\cQ}_{i,0})=0$, i.e.~that $\alpha_{t_i}$ is an isomorphism at the generic point. By Item~(3) we have $c_1(\cG_{t_i})=c_1(\cF_{t_i})$ and it follows that $\alpha_{t_i}$ is an isomorphism, contradiction. We have proved that~\eqref{tratra} holds if $\ell=0$.  

 If $\rk(\ov{\cQ}_{i,0})=0$ then 
$c_1(\ov{\cQ}_{i,\ell})$ is effective, and hence the integral in~\eqref{intpend}  is positive or zero, and if 
$\rk(\ov{\cQ}_{i,0})=r$ then $\cF_{t_i}\cong \ov{\cQ}_{i,0}$ and hence  the integral in~\eqref{intpend}  vanishes. 

By~\eqref{pincetto} and~\eqref{tantipezzi} we get that~\eqref{deltacresce} holds.

Since $\Delta(\cG\otimes\rho^{*}(\cL))=\Delta(\cG)$ the  inequality in~\eqref{deltacresce} contradicts Item~(1).
\end{proof}
\begin{proof}[Proof of Proposition~\ref{unico}]
Existence has been proved in Proposition~\ref{buonacompt}. Let $\cE$ be the vector bundle of that proposition. Then  $[\cE]\in \cM_{e}(r_0)$.

Now let $\cA$ be any $h$ slope-stable vector such that
\begin{equation}\label{bobbysolo}
\ch_i(\cA)=\ch_i(\cE),\quad \forall i\in\{0,1,2\}.
\end{equation}
 We must prove that $\cA$ is isomorphic to $\cE$.

Let $\pi\colon X\to\PP^2$ be the associated Lagrangian fibration of $[(X,h)]$. Since $[(X,h)]$ is a generic point of 
$\cN^i_e(d)^{\text{good}}$ the discriminant divisor of $\pi$ is the dual of a smooth plane sextic curve (see Proposition~\ref{torcere}), and hence it is smooth away from a finite set $B_0\subset\PP^2$.  By Item~(2) of Proposition~\ref{buonacompt} there is a  finite (possibly empty) 
 $B_1\subset\PP^2$  of $z$ such that the restriction of $\cE$ to   $\pi^{-1}(z)$  is not slope-stable.

Let  $z_0\in(\PP^2\setminus (B_0\cup B_1))$. We claim  that  $\cA_{|\pi^{-1}(z_0)}$ is isomorphic to $\cE_{|\pi^{-1}(z_0)}$. 
In fact let $T\subset\PP^2$ be a smooth curve containing $z_0$ and  intersecting transversely the discriminant 
divisor of $\pi$. Thus $Y:=\pi^{-1}(T)$ is a smooth threefold and the restriction of $\pi$ to $Y$ defines a dominant map $\rho\colon Y\to T$. We apply Lemma~\ref{oglemma} to $\cF:=\cE_{|Y}$ and $\cG:=\cA_{|Y}$ (the polarization of $T$ is the restriction  of the polarization of $X$). Let us check that the hypotheses of that lemma are satisfied: Item~(1) holds by~\eqref{bobbysolo}, Item~(2) holds by  
Item~(2) of Proposition~\ref{buonacompt}, and Item~(3)  holds by Proposition~\ref{propriostab}. In fact by that proposition  the set of $z\in\PP^2$ such that $\cE_{|\pi^{-1}(z)}$ is not isomorphic to 
$\cA_{|\pi^{-1}(z)}$ is contained in a proper closed subset $Z\subset\PP^2$, and hence it suffices to choose 
 $T$ so that it is not contained in $Z$. This shows that the hypotheses of  Lemma~\ref{oglemma} hold, and hence  we get that the restrictions  of $\cE$ and $\cA$ to any fiber of $Y\to T$ are isomorphic. In particular $\cE_{|\pi^{-1}(z_0)}$ is isomorphic to 
 $\cA_{|\pi^{-1}(z_0)}$. 

Let $z\in(\PP^2\setminus (B_0\cup B_1))$. We have proved that  $\cA_{|\pi^{-1}(z)}$ is isomorphic to 
 $\cE_{|\pi^{-1}(z)}$, and hence in particular  $\cA_{|\pi^{-1}(z)}$ and  
 $\cE_{|\pi^{-1}(z)}$ are slope-stable. Since $c_1(\cE)=c_1(\cA)$ it follows that 
 the restrictions of $\cE$ and $\cA$ to $\pi^{-1}(\PP^2\setminus (B_0\cup B_1))$  are isomorphic, see the proof of Lemma~\ref{oglemma}. By Hartogs it follows that $\cE$ is isomorphic to $\cA$.
\end{proof}
\subsection{Proofs of Theorem~\ref{unicita} and Corollary~\ref{solosolo}}\label{dimfin}
\setcounter{equation}{0}
\begin{proof}[Proof of  Theorem~\ref{unicita}]
If $r_0=1$ the result is trivially true, hence we may assume that $r_0\ge 2$. Let $\cX\to T_{e}^1$ and $\cX\to T_{e}^2$ be  complete families of  polarized 
 HK's of Type $K3^{[2]}$ such that~\eqref{divuno}, respectively ~\eqref{divdue}, holds - e.g.~the families parametrized by the relevant  open subsets of  suitable Hilbert schemes. Since $\cK^i_e$ is irreducible we may, and will, assume that $T_{e}^i$ is irreducible.  By passing to normalization if necessary we may  assume that $T_{e}^i$ is normal. For $t\in T_{e}^i$ we let $(X_t,h_t)$ be the corresponding polarized HK of Type $K3^{[2]}$. We let $m\colon T_e^i\to\cK^i_e$ be the moduli map, sending $t$ to $[(X_t,h_t)]$.
 
 Let $f\colon \cM_{e}(r_0)\to T^i_e$  be the relative moduli space that we have introduced, see~\eqref{vbrel}. 
 
 By Proposition~\ref{unico} for $t$ in   a dense subset of  $\bigcup\limits_{d\gg 0}m^{-1}(\cN_e^i(d)^{\text{good}})$ the preimage $f^{-1}(t)$ is a singleton. Since 
 $\bigcup\limits_{d\gg 0}m^{-1}(\cN_e^i(d)^{\text{good}})$ is Zariski dense in $T^i_e$ (it is the union of an infinite collection of pairwsie distinct divisors), and since
$f(\cM_{e}(r_0))$ is a  constructible subset of $T^i_e$, it follows that for  generic $t\in  T^i_e$ the fiber 
$f^{-1}(t)$ is a singleton.

Let $[\cE]$ be the unique point of $f^{-1}(t)$ for $t$ a generic point of $m^{-1}(\cN_e^i(d)^{\text{good}})$, where $d\gg 0$. Then $H^p(X_t, End^0(\cE))=0$ by 
 Proposition~\ref{acca20}. Hence the last sentence of Theorem~\ref{unicita} follows from upper semicontinuity of cohomology.
\end{proof}
\begin{proof}[Proof of  Corollary~\ref{solosolo}]
  Since  $T^i_e$ is normal, it follows by  Theorem~\ref{unicita} and Zariski's main Theorem that every fiber of $f$ is either empty or connected. 
 
 Now let $t\in T^i_e$ such that $(X_t,h_t)$ is isomorphic to $(X,h)$; we identify $X_t$ with $X$. Let $x\in f^{-1}(t)$ be the point representing the vector bundle  $\cE$. 
  Since $h^2(X,End^0(\cE))=0$,   every irreducible component of $\cM_{e}(r_0)$ containing $x$ dominates $T^i_e$. By  Theorem~\ref{unicita} there is a single such component. Hence  
 \begin{equation}\label{chichi}
  \chi(X,End^0(\cE))=\chi(X_s,End^0(\cG)),
\end{equation}
 where $s\in T^i_e$ is generic
  and  $\cG$ is the unique (up to isomorphism) $h_s$ slope-stable vector bundle on $X_s$ such that~\eqref{ele} holds (with $\cG$ replacing $\cE$). 
  
  By~\eqref{chichi} and Theorem~\ref{unicita} we get that $\chi(X,End^0(\cE))=0$. Now $h^0(X,End^0(\cE))=0$ by stability, hence $h^4(X,End^0(\cE))=0$ by Serre duality, and 
 $h^2(X,End^0(\cE))=0$ by hypothesis. It follows that $H^1(X,End^0(\cE))=0$ (notice that $H^1(X,End^0(\cE))$ is Serre dual to $H^3(X,End^0(\cE))$), and hence 
 $\{x\}$ is a component of $f^{-1}(t)$. Since   $f^{-1}(t)$ is not empty it is  connected, and hence it  equals $\{x\}$. This proves Corollary~\ref{solosolo}.  
\end{proof}
\section{Moduli of DV  varieties}\label{geodv}
\setcounter{equation}{0}
\subsection{Debarre-Voisin vector bundles}
\setcounter{equation}{0}
Let $X\subset \Gr(6,V_{10})$ be a DV variety, and let 
\begin{equation}\label{tautdv}
0\lra \cS \lra \cO_{X}\otimes V_{10}\lra \cQ\lra 0
\end{equation}
be the restriction to $X$ of the tautological exact sequence of vector bundles  on $\Gr(6,V_{10})$. Thus $r(\cS)=6$ and $r(\cQ)=4$.  
\begin{lmm}\label{classidv}
Let $X$ be a DV variety, and let  $h\in H^{1,1}_{\ZZ}(X)$ be the  Pl\"ucker polarization. 
 Then
\begin{eqnarray}
\ch_0(\cQ) & = & 4, \label{superovvio} \\
\ch_1(\cQ) & = & h, \label{ovvio} \\
\ch_2(\cQ) & = & \frac{1}{8}\left(h^2 -c_2(X)\right), \label{qui} \\
\end{eqnarray}
where $\eta_X\in H^8(X;\ZZ)$ is the fundamental class.
\end{lmm}
\begin{proof}
The first two equations are obvious, and equation~\eqref{qui} follows from the third-to-last equation on p.~83 of~\cite{devo} (beware that  $c_i=c_i(\cS^{\vee})$). 
\end{proof}
\begin{rmk}\label{stessi}
Let $(X,h)$ be a smooth DV variety. Then $q(h)=22$ and the divisibility of $h$ is $2$, i.e.~$[(X,h)]\in\cK^2_{22}$. 
If one sets $r_0=2$, $i=2$ and $e=22$, then the equations in~\eqref{ele} give for the vector bundle $\cE$ the same rank, $\ch_1$ and $\ch_2$ as  
for the quotient vector bundle $\cQ$ described above.
\end{rmk}
\begin{rmk}\label{anchecub}
Let $X$ be the variety of lines in a smooth cubic fourfold  in $\PP(V_6)$. Let 
\begin{equation}\label{aridaje}
0\lra \cS \lra \cO_{X}\otimes V_{6}\lra \cQ\lra 0
\end{equation}
be the restriction to $X$ of the tautological exact sequence of vector bundles  on $\Gr(2,V_{6})$. Hence the rank of $\cQ$ is $4$ . Let  $h\in H^{1,1}_{\ZZ}(X)$ be the  Pl\"ucker polarization. 
 Then (see the equations in~\ref{yaris}) we have
\begin{equation*}
\ch_0(\cQ)  =  4,\ \ch_1(\cQ)  =  h,\   
\ch_2(\cQ)  =  \frac{1}{8}\left(h^2 -c_2(X)\right). 
\end{equation*}
Next notice that $q(h)=6$ and the divisibility of $h$ is $2$, i.e.~$[(X,h)]\in\cK^2_{6}$. The equations above show  that 
 the Chern character of $\cQ$ is identified with the Chern character appearing in Theorem~\ref{unicita} for 
  $r_0=2$,  $i=2$ and $e=6$.  
\end{rmk}
\begin{prp}\label{piripicchio}
If $X$ is a smooth DV variety with cyclic Picard group both $\cQ$ and $\cS$ (see~\eqref{tautdv}) are   slope-stable vector bundles. 
\end{prp}
\begin{proof}
Let $h\in H^{1,1}_{\ZZ}(X)$ be the  Pl\"ucker polarization. Let us prove that $\cQ$ is $h$ slope-stable. Suppose that 
\begin{equation*}
0\to \cA\to\cQ\to\cB\to 0
\end{equation*}
 is a desemistabilizing sequence. Thus $0<r(\cA)<4$, 
\begin{equation}
\frac{c_1(\cA)\cdot h^3}{r(\cA)}=\mu(\cA)\ge \mu(\cQ):=\frac{h^4}{4}=363,
\end{equation}
and we may assume that $\cB$ is torsion-free. By  hypothesis $c_1(\cA)=xh$ for some $x\in\ZZ$, and hence $x\ge 1$. It follows that $c_1(\cB)=(1-x)h$. Since $\cQ$ is globally generated, so is $\cB$. Thus  $x=1$, i.e.~$c_1(\cB)=0$, and $\cB$ is trivial because it is globally generated. Hence $c_4(\cQ)=0$. This is a contradiction.   In fact, following the notation on p.~83 of~\cite{devo}, we let  $c_i=c_i(\cS^{\vee})$. Then (using the formulae in Equation (11) of loc.~cit.)
\begin{equation*}
c_4(\cQ)=c_1^4-3c_1^2c_2+c_2^2+2c_1 c_3-c_4=9\eta_X.
\end{equation*}
 An analogous proof gives slope-stability of $\cS$.
\end{proof}
By openness of slope-stability we also get the following result.
\begin{crl}\label{genstab}
If $X$ is a generic  DV variety both  $\cQ$ and $\cS$ are  slope-stable  vector bundles (for the Pl\"ucker polarization). 
\end{crl}
\begin{rmk}\label{cubstab}
Let $X$ be the variety of lines in a smooth cubic fourfold  in $\PP(V_6)$, and let $\cQ$ be  quotient vector bundle appearing in~\eqref{aridaje}. If  $\Pic(X)$ is cyclic one can prove that $\cQ$ is slope-stable by proceeding as in the proof of Proposition~\ref{piripicchio}, except that in the end one does not conclude by a Chern class computation (we have $c_4(\cQ)=0$).  Rather one shows directly that there is no trivial quotient $\cQ\twoheadrightarrow \cO_X$. In fact, assume there is such a quotient; then there is a non zero section of $\cQ^{\vee}$, and one gets that the latter is false by considering the dual of the exact sequence in~\ref{aridaje}. 
\end{rmk}
\begin{prp}\label{diecisez}
If $X$ is a generic DV variety, then  the map  $V_{10}\to H^0(X,\cQ)$ induced by~\eqref{tautdv} is an isomorphism. 
\end{prp}
\begin{proof}
The vector bundle $\cS$ has no global sections because  it is slope-stable (by Corollary~\ref{genstab}) with negative slope. Hence
 it suffices to prove that  $h^0(X,\cQ)=10$. We do this by considering a $K3$ surface $S$ as in Claim~\ref{eccoell} with $m_0=3$ and large odd $d_0$. The vector bundle $\cF$ on $S$ of Proposition~\ref{rigsuk} has Mukai vector $v(\cF)=(2,D,2)$.  We claim that 
 \begin{equation}\label{treacca}
h^0(S,\cF)=4,\qquad  h^1(S,\cF)=0,\qquad  h^2(S,\cF)=0.
\end{equation}
  at least for $d\gg 0$ and \lq\lq most\rq\rq\ $S$. In fact let $(S',D')$ be a generic polarized $K3$ surface with  $D'$ of square $6$. Then there exists a unique $D'$ slope-stable vector bundle $\cF'$ on $S'$ with Mukai vector $(2,D',2)$. As is easily checked
  $h^0(S',\cF')=4$. Since  moduli of  elliptic $K3$'s that we are considering are dense in the moduli space of polarized $K3$'s of degree $6$, we get that $h^0(S,\cF)=4$ for \lq\lq most\rq\rq\ $S$.
 By stability of $\cF$ we have $h^2(S,\cF)=0$. Hence also the middle equality in~\eqref{treacca} holds because  $\chi(S,\cF)=4$.

Let $\cE_0:=\cF[2]^{+}$. By definition there is a canonical isomorphism $H^0(S^{[2]},\cE_0)\cong\Sym^2 H^0(S,\cF)$ and hence 
$h^0(S^{[2]},\cE_0)=10$. From the second equality in~\eqref{treacca} we also get that $h^1(S^{[2]},\cE_0)=0$. Now let $X_0:=S^{[2]}$ and let $h_0:=h$ where $h$ is given by~\eqref{poler}. Notice that $q(h_0)=22$ and the divisibility of $h_0$ is $2$.

Let $\cX\to T$ be an analytic representative of the deformations space of $(X_0,h_0)$. Let $t\in T$; by Proposition~\ref{acca20} and Corollary~\ref{iacman} there is one and only one vector bundle  $\cE_t$ on $X_t$ which is a deformation of $\cE_0$.  By Proposition~\ref{chiave} $h_t$ is a polarization on $X_t$ for $t$ generic in $T$, and $\cE_t$ is $h_t$ slope-stable. But for $t\in T$ generic  $(X_t,h_t)$ is isomorphic to a DV variety parametrized by an analytic open subset of $\PP(\bigwedge^3 V_{10}^{\vee})$. Hence $\cE_t$ is isomorphic to the corresponding quotient DV vector bundle $\cQ_t$ on $X_t$ by Theorem~\ref{unicita} and Corollary~\ref{genstab}. Hence  $h^0(X_t,\cE_t)=h^0(X_0,\cE_0)=10$ because  $h^1(X_0,\cE_0)=0$. 
\end{proof}
\subsection{Proof of Theorem~\ref{perdivi}}
\setcounter{equation}{0}
\begin{proof}
Let $d$ be the degree of the moduli map 
\begin{equation}
\cM_{DV}\dra \cK^2_{22}.
\end{equation}
 We have $d\ge 1$ because the moduli map is dominant. We need to prove that $d=1$.  Let $[(X,h)]\in\cK^2_{22}$ be a generic point. Then there exist $[\sigma_1],\ldots,[\sigma_d]\in\PP(\bigwedge^3 V_{10}^{\vee})$ such that the corresponding polarized DV varieties $(X_1,h_1)\ldots,(X_d,h_d)$ are
smooth and all isomorphic to $(X,h)$, but the $\PGL(V_{10})$-orbits of $[\sigma_1],\ldots,[\sigma_d]$ are   pairwise distinct. Let $\cQ_i$ be the DV quotient vector bundle on $X$ determined by 
$\sigma_i$. By Corollary~\ref{genstab} each $\cQ_i$ is $h$ slope-stable, and hence  by Theorem~\ref{unicita} all the $\cQ_i$ are isomorphic to a single vector bundle $\cE$. By Proposition~\ref{diecisez}  the surjection 
$\cO_{X_i}\otimes V_{10}\twoheadrightarrow \cQ_i$ is identified with the canonical map
$\cO_{X}\otimes H^0(X,\cE)\lra \cE$. It follows that $d=1$.
\end{proof}
\begin{rmk}
 Let  $|\cO_{\PP^5}(3)| \gquot\PGL(6)\dra\cK^2_6$ be the moduli map one gets by associating to a smooth cubic $4$-fold the variety of its lines. This map is birational by Voisin's Global Torelli Theorem for cubics. Charles~\cite{charles-torelli} inverted the argument: he proved independently that the moduli map $|\cO_{\PP^5}(3)| \gquot\PGL(6)\dra\cK^2_6$ is birational and obtained Global Torelli for cubic $4$-folds from Global Torelli for HK's.  
 
Here we notice that Charles' result can also be obtained  arguing as in the proof of Theorem~\ref{perdivi}. 
\end{rmk}

\subsection{Relation with degenerate DV varieties}\label{futuro}
\setcounter{equation}{0}
In short, one may approach the subject of~\cite{dhov-journal} from the \lq\lq opposite\rq\rq\ direction. That means starting from  $(X,h,\cE)$ where $X$ is of Type $K3^{[2]}$ (either $S^{[2]}$ for a suitable $K3$, or a HK birational  to  $S^{[2]}$), $h$  is a big and nef class  such that $q(h)=22$ and  $q(h,H^2(X;\ZZ))=(2)$, and  $\cE$
is a slope-stable vector bundle (or more generally  a GM stable torsion-free sheaf)  such that~\eqref{ele} holds with $r_0=2$. Then $(X,h)$ should correspond to a degenerate $\sigma\in\bigwedge^3 V_{10}^{\vee}$ (\lq\lq degenerate\rq\rq\  means that the corresponding Debarre-Voisin variety is not smooth of dimension $4$) if $\cE$ is not as good as possible, e.g.~$h^0(X,\cE)>10$, or $h^0(X,\cE)=10$ but $\cE$ is not globally generated, or it is globally generated but the corresponding map $X\to\Gr(6,H^0(\cE))$ is not an embedding, or  $\cE$ is not locally free (one should also  take into account the possibility of getting a degenerate $\sigma$ because $h$ is not ample). An example: in the proof of Proposition~\ref{diecisez} we discussed a case in which  $h^0(X,\cE)=10$ and $\cE$ is globally generated but the corresponding map $X\to\Gr(6,H^0(\cE))$ is not an embedding. The \lq\lq inverse\rq\rq\ approach should allow to complete the discussion of the family appearing in Section~8 of~\cite{dhov-journal}.

\appendix

\section{(Semi)homogeneous vector bundles on abelian varieties}\label{mezzomo}
\subsection{Basics}
\setcounter{equation}{0}
 Let $A$ be an abelian  variety, and let $A^{\vee}:=\Pic^0(A)$ be its dual abelian variety.  For $a\in A$, let $T_a\colon A\to A$ be the translation by $a$.  For  an invertible sheaf $\xi$ on $A$, we let $[\xi]\in A^{\vee}$ be its isomorphism class. 
\begin{dfn}
  A vector bundle  $\cF$ on $A$ is \emph{homogeneous} if $T_a^{*}\cF\cong \cF$  for every $a\in A$, 
it    is \emph{semi-homogeneous} if, for every $a\in A$, there exists $[\xi]\in \Pic^0(A)$ such that $T_a^{*}\cF\cong \cF\otimes\xi$. 
\end{dfn}
\begin{prp}\label{discperp}
Let $(A,\theta)$ be a polarized abelian variety  of dimension $n$, and let   $\cF$ be a $\theta$ slope-stable vector bundle on $A$.  If
\begin{equation}\label{delper}
\int_A \Delta(\cF)\smile \theta^{n-2}=0
\end{equation}
(the condition is to be understood to be empty if $n=1$) then $\cF$ is simple semi-homogeneous. Moreover  $\Delta(\cF)=0$.
\end{prp}
\begin{proof}
Of course $\cF$ is simple because it is slope-stable.

If $n=1$ then $\cF$ is  semi-homogeneous by Atiyah's classification of simple vector bundles on elliptic curves. 

Suppose that $n\ge 2$. By the Kobayashi-Hitchin correspondence, $\cF$ has a $\theta$ Hermitian-Einstein metric, and hence so does the vector bundle $End (\cF)$.  Equation~\ref{delper} is equivalent to $\int_A c_2(End(\cF))\smile \theta^{n-2}=0$. Since $c_1(End(\cF))=0$, 
the Hermite-Einstein connection on  $End (\cF)$ is flat, see \S 4 in~\cite{kobstab}.  Hence  $End (\cF)$ is homogeneous, and thus  
$\cF$ is semi-homogeneous by Theorem 5.8 in~\cite{muksemi}. 

Since $End (\cF)$ is homogeneous, it is a direct sum of vector bundles which have  filtrations whose associated graded bundles are  direct sums of a topologically trivial line bundles - see Theorem 4.17 in~\cite{muksemi}. 
Thus  $\Delta(\cF)=c_2(End(\cF))=0$.
\end{proof}
\subsection{Rank  of semi-homogeneous vector bundles}
\setcounter{equation}{0}
The rank of a simple semi-homogeneous vector bundle with assigned first Chern class is not arbitrary.  Below we extend a result 
of Mukai, see Theorem 7.11 and Remark 7.13 in~\cite{muksemi}, so that we cover all the canonical polarizations that occur on Lagrangian fibrations in the known deformation classes of HK's with the exception of OG10.  
\begin{prp}[Mukai~\cite{muksemi}]\label{potenza}
Let $(A,\theta)$ be a  polarized abelian variety of dimension $n$. Suppose that the elementary divisors of $\theta$ are $(1,\ldots,1,d_1,d_2)$ where $d_1$ divides $d_2$.  Let $\cF$ be a simple semi-homogeneous vector bundle  on $A$ such that $c_1(\cF)=a\theta$. Then there exists a positive integer $r_0$    such that, letting $g_i:=\gcd\{r_0,d_i\}$ we have   
\begin{equation}\label{controllo}
\gcd\{r(\cF),a\}=\frac{r_0^{n-1}}{g_1\cdot g_2},\quad  r(\cF)=\frac{r_0^{n}}{g_1\cdot g_2}.
\end{equation}
\end{prp}
\begin{proof}
Let $r:=r(\cF)$ and let $c:=\gcd\{r,a\}$. We let $r_0:=\frac{r}{c}$ and  $a_0:=\frac{a}{c}$. Let us prove that~\eqref{controllo} holds. Let $\cL$ be a line bundle on $A$ such that $c_1(\cL)=a_0\theta$, and let $\varphi_{\cL}\colon A\to A^{\vee}$ be the homomomorphism  defined by  $\varphi_{\cL}(a):=T_a^{*}(\cL)\otimes\cL^{-1}$. Let $K(\cL)$ be the kernel of $\varphi_{\cL}$. Lastly, following Mukai, we let
\begin{equation}\label{eccosigma}
\Sigma(\cF):=\{[\xi]\in A^{\vee} \mid \exists \cF\overset{\sim}{\lra}\cF\otimes\xi\}.
\end{equation}
(Since we are in characteristic $0$, the above set-theoretic definition coincides with the schematic one, see Proposition~5.9 in~\cite{muksemi}.) By Corollary~7.8 in~\cite{muksemi} we have an exact sequence of groups:
\begin{equation}\label{utensile}
0 \lra A[r_0]\cap K(\cL) \lra A[r_0] \overset{\varphi_{\cL}}{\lra} \Sigma(\cF) \lra 0.
\end{equation}
Because of our hypothesis on the elementary divisors of $\theta$ we have
\begin{equation}
K(\cL)\cong(\ZZ/(a_0))^2\oplus\ldots\oplus (\ZZ/(a_0 d_1))^2\oplus (\ZZ/(a_0 d_2))^2.
\end{equation}
Since $a_0,r_0$ are coprime, it follows that $A[r_0]\cap K(\cL)\cong (\ZZ/(g_1))^2\oplus (\ZZ/(g_2))^2$. 
Thus $|\Sigma(\cF)|=\frac{r_0^{2n}}{g_1^2 \cdot g_2^2}$. 
On the other hand the cardinality of $\Sigma(\cF)$ is equal to $r^2$ by Proposition 7.1 in~\cite{muksemi}. Thus we get that  
$r=\frac{r_0^n}{g_1\cdot g_2}$. Since $c:=\frac{r}{r_0}$ it follows that $\gcd\{r,a\}=c=\frac{r_0^{n-1}}{g_1\cdot g_2}$.  
\end{proof}
\section{Polarized Lagrangian HK's of Type $K3^{[2]}$}
\subsection{Lagrangian Noether-Lefschetz loci}
\setcounter{equation}{0}
We recall that $\cK^i_e$ is the moduli space of polarized HK's of Type $K3^{[2]}$ with polarization of BBF square $e$ and divisibility given by $i$ (which is either $1$ or $2$) - see Subsection~\ref{risprin}. 
\begin{dfn}\label{ennelagr}
For $d$ a strictly positive integer let $\cN_e^i(d)\subset \cK_e^i$ be the closure of the locus parametrizing polarized HK's $(X,h)$ such that $H^{1,1}_{\ZZ}(X)$ 
contains a  saturated rank $2$ sublattice generated by $h,f$, where 
\begin{equation}\label{disez}
q(f)=0,\quad q(h, f)=d.
\end{equation}
\end{dfn}
\begin{prp}\label{unicafibr}
 Keeping notation as above, suppose in addition that $d$ is even if $i=2$, and that 
 \begin{equation}\label{chiarello}
d>5(e+1),\quad e\notdivides 2d.
\end{equation}
Then $\cN_e^i(d)$ is closed of pure codimension $1$ (in particular non empty), and  if 
$[(X,h)]\in\cN_e^i(d)$ is  generic  there is one and only one Lagrangian fibration $\pi\colon X\to\PP^2$ (modulo automorphisms of $\PP^2$) such that, letting  $f:=c_1(\pi^{*}\cO_{\PP^2}(1))$,   the equalities in~\eqref{disez} hold and the sublattice  
$\la h,f\ra\subset H^{1,1}_{\ZZ}(X)$ is saturated. 
\end{prp}
\begin{proof}
Before starting the proof we emphasize that $\cN_e^i(d)$ might have several irreducible components, and that \lq\lq generic point\rq\rq\ of $\cN_e^i(d)$ means belonging to an open \emph{dense} subset of $\cN_e^i(d)$. 
By surjectivity of the period map there exists a HK $X$ of Type $K3^{[2]}$ such that $H^{1,1}_{\ZZ}(X)=\la h,f\ra$ where 
\begin{equation}
q(h)=e,\quad \{q(h,\alpha) \mid \alpha\in H^2(X;\ZZ)\}=(i),
\end{equation}
 and the equalities in~ \eqref{disez} hold. 
There are no  $\xi\in H^{1,1}_{\ZZ}(X)$ with $-10\le q(\xi)<0$ by Lemma~\ref{nocamere} and the inequality in~\eqref{chiarello}.
It follows that the  ample cone of $X$ is equal to the intersection of $H^{1,1}_{\ZZ}(X)$ and the positive cone. Hence either $h$ or $-h$ is ample. If the former holds then $[(X,h)]\in \cN_e^i(d)$, if the latter holds, we may replace $h,f$ by $-h,-f$ respectively, and again we get that  $[(X,h)]\in\cN_e^i(d)$.
Moreover  $\cN_e^i(d)$ is closed of pure codimension $1$ because it is a Noether-Lefschetz divisor.

A straightforward computation shows  that there are exactly two primitive nef isotropic classes, namely  $f$ and $\alpha:=\frac{1}{\gcd\{d,e\}}(2d h-e f)$. 
  By our \lq\lq non divisibility\rq\rq\ hypothesis in~\eqref{chiarello}, we get that  $q(\alpha,h)=\frac{de}{\gcd\{d,e\}}$ is not equal to $d$. Hence $f$ is the unique primitive nef isotropic  class  such that $q(h, f)=d$. 
  
  By Theorem~1.3 in~\cite{markman-lagr} (see also Remark~1.8) there 
exist a Lagrangian fibration $\pi\colon X\to\PP^2$ such that  $f:=c_1(\pi^{*}\cO_{\PP^2}(1))$. By Theorem~1.2 in~\cite{matsushita-iso-div} it follows that there exists a Zariski open neighborhood $\cU$ of $[(X,h)]$ in $\cN_e^i(d)$ such that each representative $(X',h')$ of a point in $\cU$ has a Lagrangian fibration as required. 
Since the set of points of $\cN_e^i(d)$ representing $(X,h)$ such that   $\rho(X)=2$  is dense, this proves the result about existence of the required Lagrangian fibration.  

It remains to prove that if $[(X,h)]\in\cN_e^i(d)$ is generic then there is a unique isotropic class $f$ such that $q(h,f)=d$. We checked that this is the case if $\rho(X)=2$. It follows that the statement holds for the generic 
point of $\cN_e^i(d)$; the argument is similar to that given to show that $h$ is $a(\cE)$-suitable in the proof  of Proposition~\ref{propriostab}.
\end{proof}
\begin{dfn}\label{fibrass}
Suppose that~\eqref{chiarello} holds. We let $\cN_e^i(d)^0\subset\cN_e^i(d)$ be an open dense subset such that the thesis of Proposition~\ref{unicafibr} holds for any
 $[(X,h)]\in\cN_e^i(d)^0$.  For  $[(X,h)]\in\cN_e^i(d)^0$ the \emph{associated Lagrangian fibration}  $\pi\colon X\to\PP^2$  is the unique fibration (modulo automorphisms of $\PP^2$) of Proposition~\ref{unicafibr}. 
\end{dfn}
\subsection{Tate-Shafarevich twists}
\setcounter{equation}{0}
A basic example of Lagrangian fibration is constructed as follows. Let $S\to\PP^2$ be the double cover ramified over a smooth sextic curve $B$, i.e.~a polarized $K3$ surface of degree $2$. Let $\cJ(S)$ be the moduli space of rank $0$ pure $\cO_S(1)$ semistable sheaves $\xi$ with $\chi(\xi)=-1$. The generic point of $\cJ(S)$ is represented by $i_{*}\cL$, where $i\colon C\hra S$ is the inclusion of a smooth $C\in \cO_S(1)$, and $\cL$ is a line bundle of degree $0$. Then for generic $B$ every semistable sheaf is stable (the precise condition is that $B$ have no tritangents), and hence $\cJ(S)$ is smooth. In fact  it is a HK of Type $K3^{[2]}$, and 
the support map $\cJ(S)\to(\PP^2)^{\vee}$ is a Lagrangian fibration. 

A  Lagrangian fibration parametrized by a generic point of  $\cN_e^i(d)$ is related to a generic $\cJ(S)$ via a Tate-Shafarevich twist. In order to be more precise, we recall a result of Markman. 
Let $(X,h)$ be a 
representative of a generic point of $\cN_e^i(d)$. Then there is an associated  polarized $K3$ surface  $(S,D)$ of degree $2$, and  moreover  $(S,D)$ is  generic  - see Subsection~4.1 in~\cite{markman-lagr}.
\begin{prp}\label{torcere}
Keep the hypotheses of Proposition~\ref{unicafibr}. Let $[(X,h)]$ be a generic point of $\cN_e^i(d)$. Let $\pi\colon X\to\PP^2$ and  $S$ be the associated Lagrangian fibration and $K3$ surface of degree $2$ respectively. Then $X$ is isomorphic to a Tate-Shafarevich twist of $\cJ(S)\to(\PP^2)^{\vee}$ via an identification $\PP^2\overset{\sim}{\to}(\PP^2)^{\vee}$. 
\end{prp}
\begin{proof}
Suppose first that $\rho(X)=2$. Then, as shown in the proof of Proposition~\ref{unicafibr}, the  ample cone of $X$ is equal to the positive cone (because of the inequality in~\eqref{chiarello}), and hence every  bimeromorphic map $X\dra X'$, where $X'$ is a $HK$,  is actually an isomorphism. It follows that $X$ is isomorphic to a Tate-Shafarevich twist of $\cJ(S)\to(\PP^2)^{\vee}$ by Theorem~7.13 in~\cite{markman-lagr}. The result follows from this because the locus in $\cN_e^i(d)$ parametrizing $(X,h)$ such that $\rho(X)=2$ is dense.
\end{proof}
Let $X\to\PP^2$ be as in Proposition~\ref{torcere}, and let
$\Pic^0(X/\PP^2)$  be the relative Picard scheme  (notice that all fibers of $X\to\PP^2$ are irreducible by Proposition~\ref{torcere}). 
If $U\subset\PP^2$ is the open dense set of regular values of $X\to\PP^2$ and 
 $z\in U$, the fiber of $\Pic^0(X/\PP^2)\to\PP^2$ over $z$ is an abelian surface $A_z$ and
the fundamental group $\pi_1(U,z)$ acts by monodromy on the subgroup $A_{z,tors}$ of torsion points. 
\begin{crl}\label{modinv}
Keep hypotheses as above, and suppose that 
 $V\subset A_z[r_0^2]$ is a coset (of a subgroup) of cardinality $r_0^4$ invariant under the action of monodromy. 
 Then $V= A_z[r_0]$.
\end{crl}
\begin{proof}
Let $S$ be the polarized $K3$ surface of degree $2$ associated to $X$ following Markman, and let $S\to\PP^2$ be the double cover ramified over a sextic curve $B$. 
Let $\cJ(S)_0\subset \cJ(S)$ be the open dense subset of   smooth points (i.e.~smooth points of $\cJ(S)$ with surjective differential) of the map $\cJ(S)\to(\PP^2)^{\vee}$. 
By Proposition~\ref{torcere}, $\Pic^0(X/\PP^2)\to\PP^2$ is isomorphic to $\cJ(S)_0\to(\PP^2)^{\vee}$, for a certain identification $\PP^2\overset{\sim}{\lra} (\PP^2)^{\vee}$. Under this identification $z\in\PP^2$ corresponds to a line $R\in (\PP^2)^{\vee}$
 transverse to $B$, and the corresponding Lagrangian fiber $A_z$   is the Jacobian of the curve $C$ which is the double cover of $R$ ramified over $R\cap B$. Hence we have a natural isomorphism 
\begin{equation}\label{toromo}
H_1(C;\QQ)/H_1(C;\ZZ)\overset{\sim}{\lra}A_{z,tors},
\end{equation}
 and the identification is compatible with the monodromy actions.

First we prove the result under the assumption that $V$ is a subgroup $G$. By the structure theorem for finite
abelian groups $G\cong \ZZ/(d_1)\oplus\ldots \oplus \ZZ/(d_r)$, where $r\le 4$ (because 
$A_z[r_0^2]\cong \ZZ/(r_0^2)^{\oplus 4}$) and $d_i| r_0^2$ for all $i$.   Since 
the monodromy action on  $H_1(C;\ZZ)$ is transitive on non zero elements, 
it follows from the isomorphism in~\eqref{toromo} that $r=4$ and $d_1=\ldots=d_r$. Thus $d_i=r_0$ for all $i\in\{1,\ldots,4\}$ because  $|G|=r_0^4$. This proves the result under the assumption that $V$ is a subgroup.

Now let $V$ be a translate of a group $G$. Then $G=\{a-b \mid a,b\in V\}$, and hence $G$ is also  invariant for the monodromy action. Thus $G= A_z[r_0]$, and hence the coset
$V$ gives a point of the quotient $A_z[r_0^2]/A_z[r_0]$ which is invariant for the monodromy action. By the isomorphism in~\eqref{toromo} it follows that $0$ is the unique invariant element, and hence $V= A_z[r_0]$.
\end{proof}
%

 \bibliography{ref-vbs-on-hks}
 \end{document}